\documentclass[10pt,oneside,a4paper]{amsart}
\usepackage{amsmath,amsthm, xy,latexsym,verbatim}
\usepackage{amsxtra}
\usepackage{mathabx}
\usepackage{pxfonts}
\usepackage[psamsfonts]{amssymb}
\usepackage[applemac]{inputenc}
\usepackage{fancyhdr}
\usepackage{mathrsfs}
\usepackage[bookmarks,colorlinks]{hyperref}
\fancypagestyle{plain}{
	\fancyhead{}
	
}

\newtheorem{definition}{Definition}[section]
\newtheorem{theorem}[definition]{Theorem}
\newtheorem{proposition}[definition]{Proposition}
\newtheorem{lemma}[definition]{Lemma}

\newtheorem{remark}[definition]{Remark}
\newtheorem{example}[definition]{Example}

\newcommand\dcut{d\llap {\raisebox{.9ex}{$\scriptstyle-\!$}}}

\def\supp{ {\operatorname{supp}} }

\newcommand{\x}{\langle x\rangle}
\newcommand{\csi}{\langle \xi \rangle}

\def\R {\mathbb{R}}       
\def\N {\mathbb{N}}       
\def\Z {\mathbb{Z}}

\def\SG {{S}}

\def\ds{\displaystyle}

\def\cS{{\mathcal S}}

\def\1{\lambda}
\def\2{\Sigma_{2}}

\def\vp{\varphi}

\def\<{{\langle}}
\def\>{{\rangle}}
\def\norm#1{{\langle} #1 {\rangle}}

\def\SX{ {\mathcal{S}} }

\newcommand{\afrac}[2]{\genfrac{}{}{0pt}{1}{#1}{#2}} 
\newcommand{\beqsn}{\arraycolsep1.5pt\begin{eqnarray*}}
\newcommand{\eeqsn}{\end{eqnarray*}\arraycolsep5pt}
\newcommand{\beqs}{\arraycolsep1.5pt\begin{eqnarray}}
\newcommand{\eeqs}{\end{eqnarray}\arraycolsep5pt}
\newcommand*{\dbar}{\mbox{\dj}}	

\def\Op{ {\operatorname{Op}} }
\newcommand{\Ph}{\mathcal P}
\newcommand{\Phr}{\mathcal P_r}
\def\fy{\varphi}

\author{Alessia Ascanelli}

\address{Dipartimento di Matematica, Università degli Studi di Ferrara, via Machiavelli 30, 44121 Ferrara, Italy}

\email{alessia.ascanelli@unife.it}

\author{Sandro Coriasco}

\address{Dipartimento di Matematica ``G. Peano'', Università degli Studi di Torino, via Carlo Alberto 10. 10123 Torino,  Italy}

\email{sandro.coriasco@unito.it}

\title[FIO algebra and fundamental solution to SG hyperbolic systems]
{Fourier Integral Operators Algebra and\\ Fundamental Solutions to hyperbolic systems
\\with polynomially bounded coefficients on $\R^n$}



\begin{document}

\begin{abstract}
We study the composition of an arbitrary number of Fourier integral operators $A_j$,
$j=1,\dots,M$, $M\ge 2$, defined through symbols belonging to the so-called SG 
classes. We give conditions ensuring that the composition $A_1\circ\cdots\circ A_M$
of such operators still belongs to the same class.  Through this, we are then able 
to show well-posedness in weighted Sobolev spaces for first order hyperbolic 
systems of partial differential equations with coefficients in SG classes, by constructing 
the associated fundamental solutions. These results expand the existing theory
for the study of the properties ``at infinity'' of the solutions to hyperbolic Cauchy
problems on $\R^n$ with polynomially bounded coefficients.  
\end{abstract}

\maketitle

\tableofcontents


\section{Introduction}\label{sec:intro}
\setcounter{equation}{0}
%
%
We deal with a class of Fourier integral operators globally defined on $\R^n$, namely, the
SG Fourier integral operators (SG FIOs, for short, in the sequel), that is, the class of FIOs defined through symbols
belonging to the so-called SG classes. 

The class $\SG ^{m,\mu}(\R^{2n})$ of SG symbols of order $(m,\mu) \in \R^2$ is given by all the functions 
$a(x,\xi) \in C^\infty(\R^n\times\R^n)$
with the property
that, for any multiindices $\alpha,\beta \in \Z_+^n$, there exist
constants $C_{\alpha\beta}>0$ such that the conditions 
\begin{equation}
	\label{eq:disSG}
	|D_\xi^{\alpha} D_x^{\beta} a(x, \xi)| \leq C_{\alpha\beta} 
	\x^{m-|\beta|}\csi^{\mu-|\alpha|},
	\qquad (x, \xi) \in \R^n \times \R^n,
\end{equation}
hold. Here $\langle x \rangle=(1+|x|^2)^{1/2}$ when $x\in\R^n$, and $\Z_+$
is the set of non-negative integers.
These classes, together with corresponding
classes of pseudo-differential operators 
$\Op (\SG ^{m,\mu})$,
were first introduced in
the '70s by H.O.~Cordes \cite{CO} and C.~Parenti
\cite{PA72}, see also R.~Melrose \cite{ME}. They form a
graded algebra with respect to composition, i.e.,
$$
\Op (\SG ^{m_1,\mu _1})\circ \Op (\SG ^{m_2,\mu _2})
\subseteq \Op (\SG ^{m_1+m_2,\mu _1+\mu _2}),
$$
whose residual
elements are operators with symbols in
\[
	 \SG ^{-\infty,-\infty}(\R^{2n})= \bigcap_{(m,\mu) \in \R^2} \SG ^{m,\mu} (\R^{2n})
	 =\cS(\R^{2n}),
\]
that is, those having kernel in $\cS(\R^{2n})$, continuously
mapping  $\cS^\prime(\R^n)$ to $\cS(\R^n)$.

\par

Operators in $\Op (\SG ^{m,\mu})$
are continuous on $\cS(\R^n)$, and extend uniquely to
continuous operators on $\cS^\prime(\R^n)$ and from
$H^{s,\sigma}(\R^n)$ to $H^{s-m,\sigma-\mu}(\R^n)$,
where $H^{r,\varrho}(\R^n)$,
$r,\varrho \in \R$, denotes the weighted Sobolev (or Sobolev-Kato) space
\begin{equation*}
  	H^{r,\varrho}(\R^n)= \{u \in \cS^\prime(\R^{n}) \colon \|u\|_{r,\varrho}=
	\|{\norm .}^r\norm D^\varrho u\|_{L^2}< \infty\}.
\end{equation*}

\par

An operator $A=\Op(a)$, is called \emph{elliptic}
(or $\SG ^{m,\mu}$-\emph{elliptic}) if
$a\in \SG ^{m,\mu} (\R^{2n})$ and there
exists $R\ge0$ such that
%
\[
	C\norm{x}^{m} \norm{\xi}^{\mu}\le |a(x,\xi)|,\qquad 
	|x|+|\xi|\ge R,
\] 
for some constant $C>0$. An elliptic SG  operator $A \in \Op (\SG ^{m,\mu})$ admits a
parametrix $P\in \Op (\SG ^{-m,-\mu})$ such that
\[
PA=I + K_1, \quad AP= I+ K_2,
\]
for suitable $K_1, K_2\in\Op(\SG^{-\infty,-\infty}(\R^{2n}))$, where $I$ denotes the identity operator. 
In such a case, $A$ turns out to be a Fredholm
operator on the scale of functional spaces $H^{r,\varrho}(\R^n)$,
$r,\varrho\in\R$.

\par

In 1987, E.~Schrohe \cite{Sc} introduced a class of non-compact
manifolds, the so-called SG manifolds, on which a version of
SG calculus can be defined.
Such manifolds admit a finite atlas, whose changes of
coordinates behave like symbols of order $(0,1)$ (see \cite{Sc}
for details and additional technical hypotheses). A relevant
example of  SG manifolds are the manifolds with
cylindrical ends, where also the concept of classical SG
operator makes sense,
see, e.{\,}g. \cite{BC10a,CoMa13,CoSc14,ES97,MP02,ME}.
With $\widehat{u}$ denoting the Fourier transform of $u\in\SX(\R^n)$, given by
\begin{equation}\label{eq:tfu}
	\widehat{u}(\xi)=\int e^{-ix\cdot\xi}u(x)\,dx,
\end{equation}
for any $a\in S^{m,\mu}(\R^{2n})$, $\varphi\in\Ph$ 
-- the set of SG phase functions, see Section \ref{sec:sgcalc}
below --, the SG FIOs are defined, for $u\in\SX(\R^{n})$, as
\begin{align}
\label{eq:typei}
u\mapsto (\Op _\fy (a)u)(x)&= (2\pi )^{-n}\int
e^{i\fy (x,\xi )} a(x,\xi )
\widehat u(\xi )\, d\xi ,
\end{align}
and
\begin{align}
\label{eq:typeii}
u\mapsto (\Op^*_\fy (a)u)(x)&= (2\pi )^{-n}\iint
e^{i(x\cdot \xi -\fy (y,\xi ))} \overline {a(y,\xi )} u(y)\, dyd\xi.
\end{align}
Here the operators $\Op _\fy (a)$ and
$\Op _\fy ^*(a)$ are sometimes called SG FIOs
of type I and type II, respectively, with symbol $a$ and
SG phase function $\fy$. Note that a type II operator satisfies
$\Op^*_\fy(a)=\Op _\fy (a)^*$, that is, it is the formal $L^2$-adjoint of the 
type I operator $\Op_\fy(a)$. 

The analysis of SG FIOs started in 
\cite{Coriasco:998.1}, where composition results with the corresponding classes of pseudodifferential operators,
and of SG FIOs of type I and type II with regular phase functions, have been proved, as well as the basic
continuity proprties in $\SX(\R^n)$ and $\SX^\prime(\R^n)$ of operators in the class. A version of the Asada-Fujiwara 
$L^2(\R^n)$-continuity theorem was also proved there, for operators $\Op_\fy(a)$ 
with symbol $a\in\SG^{0,0}(\R^{2n})$ and regular SG phase function $\fy\in\Phr$, see Definition \ref{def:phaser}.
Applications to SG hyperbolic Cauchy problems were initially given in \cite{Coriasco:998.2, Coriasco:998.3}.

Many authors have, since then, expanded the SG FIOs theory in various directions. To mention a few, see, e.g.,
G.D. Andrews \cite{Andrews}, M. Ruzhansky, M. Sugimoto \cite{RuSu}, 
E. Cordero, F. Nicola, L Rodino \cite{CorNicRod1},
and the recent works by S. Coriasco and M. Ruzhansky \cite{CoRu}, 
S. Coriasco and R. Schulz \cite{CoSc13,CoSc14}. Concerning applications to SG hyperbolic problems 
and propagation of singularities, see, e.g., A. Ascanelli and M. Cappiello \cite{AC06,AC08, AC10},
M. Cappiello \cite{Ca04}, S. Coriasco, K. Johansson, J, Toft \cite{CJT4}, S. Coriasco, L. Maniccia \cite{CoMa}. 
Concerning applications to anisotropic evolution equations of Schr\"odinger type see, e.g., 
A. Ascanelli, M. Cappiello \cite{AC13}.

Here our aim is to expand the results in \cite{Coriasco:998.1, Coriasco:998.2}, through the study of the composition 
of $M\geq 2$ SG FIOs $A_j:=\Op_{\varphi_j}(a_j)$ with \textit{regular} SG phase functions 
$\varphi_j\in\Phr(\tau_j)$ -- see Definition \ref{def:phaser} below -- and symbols $a_j\in\SG^{m_j,\mu_j}(\R^{2n})$, $j=1,\ldots,M$. 
To our best knowledge, the composition of SG FIOs with different phase functions of the type that we consider in this paper
has not been studied by other authors.

First, we shall prove, under suitable assumptions, the existence of a SG phase function $\phi\in\Phr(\tau)$, 
called the \textit{multi-product} of the SG phase functions $\varphi_1,\ldots,\varphi_M$, and of a symbol 
$a\in\SG^{m,\mu}(\R^{2n})$, with 
$m:=m_1+\cdots+m_M$, $\mu:=\mu_1+\cdots+\mu_M$, such that 
\begin{equation}\label{eq:fiomltpr}
	A=\Op_\phi(a):=A_1\circ\cdots\circ A_M,
\end{equation}
see Theorem \ref{thm:main} below for the precise statement.

Subsequently, we apply such result to study a class of hyperbolic Cauchy pro\-blems. We focus on first order systems of partial differential equations of hyperbolic type with $(t,x)-$depending coefficients in SG classes. By means of Theorem \ref{thm:main}, we construct the fundamental solution $\{E(t,s)\}_{0\leq s\leq t\leq T}$ to the system. The existence of the fundamental solution provides, via Duhamel's formula, existence and uniqueness of the solution to the system, for any given Cauchy data in the weighted Sobolev spaces $H^{r,\varrho}(\R^n)$. A remarkable feature, typical for these classes of hyperbolic problems, is the
\textit{well-posedness with loss/gain of decay at infinity}, observed for the first time in \cite{AC06}, see also Section \ref{sec:fundsol} below.
We need these results in the study of certain stochastic equations, 
which will be treated in the forthcoming paper \cite{ACS15}.

This paper is organized as follows. Section \ref{sec:sgcalc} is devoted to fixing notation 
and recalling some basic definitions and known results
on SG symbols and Fourier integral operators, which will be used throughout the paper. 
In Section \ref{sec:mpsgphf} we perform the first step of the proof of our main result, Theorem \ref{thm:main},
defining and studying the multi-product of $M\ge2$ regular SG phase functions. In Section \ref{sec:sgfioprod} we 
prove Theorem \ref{thm:main}, showing the existence, under suitable hypotheses,
of $\phi\in\Phr$ and $a\in \SG^{m,\mu}$ such that \eqref{eq:fiomltpr} holds.
Finally, in Section \ref{sec:fundsol} we obtain the fundamental solution to  
SG hyperbolic first order systems.

\section*{Acknowledgements}
\noindent
The authors were supported by the INdAM-GNAMPA grant 
``Equazioni Differenziali a Derivate Parziali di Evoluzione e Stocastiche'' (Coordinator: S. Coriasco, Dep. of Mathematics ``G. Peano'', University of Turin).

\section{SG symbols and Fourier integral operators}\label{sec:sgcalc}
\setcounter{equation}{0}
%
%
In this section we fix some notation and recall some of the results proved in \cite{Coriasco:998.1}, which will be used
below. 
SG pseudodifferential operators $a(x,D)=\Op(a)$ can be introduced by means of the usual left-quantization
\[
	(\Op(a)u)(x)= (2 \pi)^{-n}\int e^{i x \cdot \xi} a(x, \xi) \widehat{u}(\xi) d\xi, \quad u\in\SX(\R^n),
\]
with $\widehat{u}$ the Fourier transform of $u$ defined in \eqref{eq:tfu},
starting from symbols $a(x,\xi) \in C^\infty(\R^n\times\R^n)$ satisfying \eqref{eq:disSG}.
Symbols of this type belong to the class denoted by $\SG^{m,\mu}(\R^{2n})$, and the corresponding operators 
constitute the class $\Op(\SG^{m,\mu}(\R^{2n}))$. In the sequel we will often simply write 
$\SG^{m,\mu}$, fixing the dimension of the base space to $n$. For $m,\mu\in\R$, $l\in\Z_+$, $a\in\ S^{m,\mu}$,
the quantities
\[
	\vvvert a \vvvert^{m,\mu}_l 
	= 
	\max_{|\alpha+\beta|\le l}\sup_{x,\xi\in\R^n}\norm{x}^{-m+|\alpha|} 
	                                                                     \norm{\xi}^{-\mu+|\beta|}
	                                                                    | \partial^\alpha_x\partial^\beta_\xi a(x,\xi)|
\]
are a family of seminorms, defining  the Fr\'echet topology of $S^{m,\mu}$. The continuity properties of
the elements of $\Op(S^{m,\mu})$ on the scale of spaces $H^{r,\rho}$, $m,\mu,r,\rho\in\R$, is expressed 
more precisely in the next Theorem \ref{thm:sobcont} (see \cite{CO} and the references quoted therein 
for the result on more general classes of SG type symbols).
\begin{theorem}\label{thm:sobcont}
	Let $a\in S^{m,\mu}(\R^n)$, $m,\mu\in\R$. Then, for any $r,\rho\in\R$, 
	$\Op(a)\in\mathcal{L}(H^{r,\rho}(\R^n),H^{r-m,\rho-\mu}(\R^n))$, and there exists a constant $C>0$,
	depending only on $n,m,\mu,r,\rho$, such that
	\begin{equation}\label{eq:normsob}
		\|\Op(a)\|_{\mathcal{L}(H^{r,\rho}(\R^n), H^{r-m,\rho-\mu}(\R^n))}\le 
		C\vvvert a \vvvert_{\left[\frac{n}{2}\right]+1}^{m,\mu},
	\end{equation}
	where $[s]$ denotes the integer part of $s\in\R$.
\end{theorem}
We now introduce the class of SG phase functions.
Here and in what follows, $A\asymp B$ means that $A\lesssim B$ and $B\lesssim A$,
where $A\lesssim B$ means that $A\le c\cdot B$, for a suitable constant $c>0$.
\begin{definition}[$SG$ phase function]\label{def:phase}
A real valued function $\varphi\in C^\infty(\R^{2n})$  belongs to the class $\mathcal P$ of SG phase functions if it satisfies the following conditions:
\begin{enumerate}
\item $\varphi\in S^{1,1}(\R^{2n})$;
\item $\<\varphi'_x(x,\xi)\>\asymp\<\xi\>$ as $|(x,\xi)|\to\infty$;
\item $\<\varphi'_\xi(x,\xi)\>\asymp\<x\>$ as $|(x,\xi)|\to\infty$.	
\end{enumerate}
\end{definition}

Functions of class $\mathcal P$ are those used in the construction of the SG FIOs calculus. 
The SG FIOs of type I and type II, $\Op_\fy(a)$ and
$\Op^*_\fy(b)$, are defined as in \eqref{eq:typei} and \eqref{eq:typeii}, respectively, with $\fy\in\Ph$ and $a,b\in S^{m,\mu}$.
The next Theorem \ref{thm:compi} about composition between SG pseudodifferential operators and SG FIOs was originally proved in \cite{Coriasco:998.1}, see also \cite{CJT4, CoPa, CT14}.
\begin{theorem}\label{thm:compi}
Let $\fy\in\Ph$ and assume $p\in S^{t,\tau}(\R^{2n})$, $a,b\in S^{m,\mu}(\R^{2n})$. Then,
\begin{align*}
\Op (p)\circ \Op _\fy (a) &= \Op _\fy (c_1+r_1) = \Op _\fy (c_1) \mod \Op (S^{-\infty,-\infty}(\R^{2d})),
\\[1ex]
\Op (p)\circ \Op ^* _\fy(b) &= \Op ^* _\fy(c_2+r_2) = \Op^* _\fy (c_2) \mod \Op (S^{-\infty,-\infty}(\R^{2d})),
\\[1ex]
\Op _\fy (a) \circ \Op (p) &= \Op _\fy (c_3+r_3) = \Op _\fy (c_3)  \mod \Op (S^{-\infty,-\infty}(\R^{2d})),
\\[1ex]
\Op _\fy ^*(b) \circ \Op (p) &= \Op ^* _\fy(c_4+r_4) = \Op^* _\fy (c_4)\mod \Op (S^{-\infty,-\infty}(\R^{2d})),
\end{align*}
for some $c_j\in S^{m+t,\mu+\tau}(\R^{2n})$, $r_j\in S^{-\infty,-\infty}(\R^{2d})$, $j=1,\dots ,4$.
\end{theorem}
To obtain the composition of SG FIOs of type I and type II, some more hypotheses are needed, leading to the definition
of the classes $\Phr$ and $\Phr(\tau)$ of regular SG phase functions.
\begin{definition}[Regular $SG$ phase function]\label{def:phaser}
Let $\tau\in [0,1)$ and $r>0$. A function $\varphi\in\mathcal P$ belongs to the class $\mathcal P_r(\tau)$ if it satisfies the following conditions:
\begin{enumerate}
\item $\vert\det(\varphi''_{x\xi})(x,\xi)\vert\geq r$, $\forall (x,\xi)$;
\item the function  $J(x,\xi):=\vp(x,\xi)-x\cdot\xi$ is such that
\beqs\label{hyp}
\ds\sup_{\afrac{x,\xi\in\R^n}{|\alpha+\beta|\leq 2}}\frac{|D_\xi^\alpha D_x^\beta J(x,\xi)|}{\x^{1-|\beta|}\<\xi\>^{1-|\alpha|}}\leq \tau.
\eeqs
\end{enumerate}
If only condition (1) holds, we write $\fy\in\Phr$.
\end{definition}
\begin{remark}
Notice that condition \eqref{hyp} means that $J(x,\xi)/\tau$ is bounded with constant $1$ in $S^{1,1}$. Notice also that condition (1) in Definition \ref{def:phaser} is authomatically fulfilled when condition (2) holds true for a sufficiently small $\tau\in[0,1)$.
\end{remark}

For $\ell\in\N$, we also introduce the seminorms
\[
	\|J\|_{2,\ell}:=
	\ds\sum_{2\leq |\alpha+\beta|\leq 2+\ell}\sup_{(x,\xi)\in \R^{2n}}
	\ds\frac{|D_\xi^\alpha D_x^\beta J(x,\xi)|}{\x^{1-|\beta|}\<\xi\>^{1-|\alpha|}},
\]
and
\[
\|J\|_\ell:=\ds\sup_{\afrac{x,\xi\in\R^n}{|\alpha+\beta|\leq 1}}\frac{|D_\xi^\alpha D_x^\beta J(x,\xi)|}{\x^{1-|\beta|}\<\xi\>^{1-|\alpha|}}+\|J\|_{2,\ell}.
\]
We notice that $\varphi\in\mathcal P_r(\tau)$ means that (1) of Definition \ref{def:phaser} and $\|J\|_0\leq \tau$ hold, and then we define the following subclass of the class of regular SG phase functions:
\begin{definition}\label{def:phaserell}
Let $\tau\in [0,1)$, $r>0$, $\ell \geq 0$. A function $\varphi$ belongs to the class $\mathcal P_r(\tau,\ell)$ if $\varphi\in\mathcal P_r(\tau)$ and $\|J\|_\ell\leq \tau$ for the corresponding $J$.
\end{definition}

Theorem \ref{thm:compii} below shows that the composition of SG FIOs of type I and type II 
with the same regular SG phase functions is a SG pseudodifferential operator.
\begin{theorem}\label{thm:compii}
Let $\fy\in\Phr$ and assume $a\in S^{m,\mu}(\R^{2n})$, $b\in S^{t,\tau}(\R^{2n})$. Then,
\begin{align*}
\Op _\fy(a)\circ \Op _\fy^*(b)= \Op(c_5+r_5)=\Op(c_5) \mod \Op (S^{-\infty,-\infty}),
\\[1ex]
\Op _\fy ^*(b)\circ \Op _\fy (a)=\Op(c_6+r_6)=\Op(c_6) \mod \Op (S^{-\infty,-\infty}),
\end{align*}
for some $c_j\in S^{m+t,\mu+\tau}(\R^{2n})$, $r_j\in S^{-\infty,-\infty}(\R^{2d})$, $j=5,6$.
\end{theorem}
Furthermore, asymptotic formulae can be given for $c_j$, $j=1,\dots ,6$,
in terms of $\fy$, $p$, $a$ and $b$, see \cite{Coriasco:998.1}. A generalization of
Theorems \ref{thm:compi} and \ref{thm:compii} to operators defined by means of broader, generalized
SG classes was proved in \cite{CJT4, CT14}, together with similar asymptotic expansions, 
studied by means of the criteria obtained in \cite{CoTo}.
\begin{remark}\label{rem:asymp}
In particular, in Section \ref{sec:fundsol} we will make use of the following (first order) expansion of the symbol of $c_1$, coming from \cite{Coriasco:998.1}:
\beqsn
c_1(x,\xi)=p(x,\varphi'_x(x,\xi))a(x,\xi)+s(x,\xi),\quad s\in S^{m+t-1,\mu+\tau-1}(\R^{2n}).
\eeqsn
\end{remark}
Finally, when $a\in S^{m,\mu}$ is elliptic and $\fy\in\Phr$, the corresponding SG FIOs admit a parametrix, that is, 
there exist $b_1,b_2\in S^{-m,-\mu}$ such that
\begin{align}
\label{eq:parami}
\Op _\fy(a)\circ \Op _\fy^*(b_1)= \Op _\fy^*(b_1)\circ\Op _\fy(a) &= I \mod \Op (S^{-\infty,-\infty}),
\\[1ex]
\label{eq:paramii}
\Op _\fy ^*(a)\circ \Op _\fy (b_2)=\Op_\fy(b_2)\circ\Op^*_\fy(a) &= I \mod \Op (S^{-\infty,-\infty}),
\end{align}
where $I$ is the identity operator, see again \cite{Coriasco:998.1,CJT4,CT14}.

In this paper we extend the existing theory of SG FIOs, dealing with the composition of SG FIOs of type I 
with different phase functions. We then apply it to compute 
the fundamental solution to SG hyperbolic systems with coefficients of polynomial growth. 

The following result is going to be used in Sections \ref{sec:mpsgphf} and \ref{sec:fundsol}.
Given a symbol $a\in C([0,T]; S^{\epsilon,1})$ with $\epsilon\in [0,1]$, let us consider the eikonal equation
\beqs\label{eik}
\begin{cases}
\partial_t\varphi(t,s,x,\xi)=a(t,x,\varphi'_x(t,s,x,\xi)),& t\in [0,T_0]
\\
\varphi(s,s,x,\xi)=x\cdot\xi,& s\in [0,T_0],
\end{cases}
\eeqs
with $0<T_0\leq T$. By an extension of the theory developed  in \cite{Coriasco:998.2}, it is possible to
prove that the following Proposition \ref{trovala!} holds true.
\begin{proposition}\label{trovala!}
For any small enough $T_0\in[0,T]$, 
equation \eqref{eik} admits a unique solution $\varphi\in C^1([0,T_0]^2_{t,s},S^{1,1}(\R^n_{x,\xi}))$,
satisfying $J\in C^1([0,T_0]^2_{t,s},S^{\epsilon,1}(\R^n_{x,\xi}))$ and
\beqs\label{eiks}
\partial_s\varphi(t,s,x,\xi)=-a(s,\varphi'_\xi(t,s,x,\xi),\xi),
\eeqs
for any $t,s\in[0,T_0]$. Moreover, 
for every $h\geq 0$ there exists $c_h\geq 1$ and $T_h\in[0,T_0]$ such that $\varphi(t,s,x,\xi)\in\mathcal P_r(c_h|t-s|)$, 
with $\| J\|_{2,h}\leq c_h |t-s|$ for all $0\leq s\leq t\leq T_h$.
\end{proposition}

\noindent
In the sequel we will sometimes write $\varphi_{ts}(x,\xi):=\varphi(t,s,x,\xi)$, for a solution $\varphi$ of \eqref{eik}.

\section{Multiproducts of SG phase functions}\label{sec:mpsgphf}
\setcounter{equation}{0}
%
%
The first step in our construction is to define the multi-product of regular SG phase functions and to analyze its properties, which we perform in the present section, following mainly \cite{Kumano-go:1}.
	
Let us consider a sequence $\{\vp_j\}_{j\geq1}$ of regular SG phase functions $\vp_j(x,\xi)\in \mathcal P_r(\tau_j)$ with \beqs\label{1/4}
\ds\sum_{j=1}^\infty \tau_j=:\tau_0<1/4.
\eeqs

By Definition \ref{def:phaser} and assumption \eqref{1/4} we have that the sequence $\{J_k(x,\xi)/\tau_k\}_{k\geq 1}$ is bounded in $S^{1,1}$ and for every $\ell\in\N$
that there exists a constant $c_\ell>0$ such that 
\beqs\label{marr}
\|J_k\|_{2,\ell}\leq c_\ell\tau_k\quad\ {\rm and}\quad  \ds\sum_{k=1}^\infty\|J_k\|_{2,\ell}\leq c_\ell\tau_0.
\eeqs
Notice that from \eqref{hyp} we have $c_0=1.$ This will be useful in the proof of Theorem \ref{stime} at the end of the present section.
\begin{example}
A simple realization of a sequence $\{\vp_j\}_{j\geq1}$  satisfying \eqref{1/4} and \eqref{hyp} can be obtained using the phase function $\varphi(t,s,x,\xi)$ solving the eikonal equation \eqref{eik}. Indeed, it is sufficient to take a partition
\[s=t_{\ell+1}\leq t_\ell\leq\cdots\leq t_1\leq t_0=t,\]
of the interval $[s,t]$ and define
\[\varphi_j(x,\xi)=
\begin{cases}
\varphi(t_{j-1}, t_j,x,\xi) & 1\leq j\leq \ell+1
\\
x\cdot\xi & j\geq\ell+2.
\end{cases}\]
In fact, from Proposition \ref{trovala!} we know that $\varphi_j\in\mathcal P_r(\tau_j)$ with $\tau_j=c_0(t_{j-1}-t_j)$ for $1\leq j\leq \ell+1$ and with $\tau_j=0$ for $j\geq \ell+2$. Condition \eqref{1/4} is fulfilled if we choose $T_0$ small enough, since
\[\ds\sum_{j=1}^\infty \tau_j=\ds\sum_{j=1}^{\ell+1} c_0(t_{j-1}-t_j)=c_0(t-s)\leq c_0T_0<\frac14\]
if $T_0<(4c_0)^{-1}$. Moreover, again from Proposition \ref{trovala!}, we know that $\| J_j\|_{2,0}\leq c_0 |t_j-t_{j-1}|=\tau_j$ for all $1\leq j\leq \ell+1$ and $J_j=0$ for $j\geq \ell+2,$ so each one of the $J_j$ satisfies \eqref{hyp}. 
\end{example}

With a fixed integer $M\geq 1$, we denote 
\begin{align*}
	(X,\Xi)&=(x_0,x_1,\ldots,x_M,\xi_1,\ldots,\xi_M,\xi_{M+1}):=(x,T,\Theta,\xi),
	\\
	(T,\Theta)&=(x_1,\ldots,x_M,\xi_1,\ldots,\xi_M),
\end{align*}
and define the function of $2(M+1)n$ real variables
\beqs\label{serveilnome}
\psi(X,\Xi):=\ds\sum_{j=1}^M\left(\vp_j(x_{j-1},\xi_j)-x_j\cdot\xi_j\right)+\vp_{M+1}(x_M,\xi_{M+1}).
\eeqs
For every fixed $(x,\xi)\in\R^{2n}$, the critical points $(Y,N)=(Y,N)(x,\xi)$ 
of the function of $2Mn$ variables $\widetilde{\psi}(T,\Theta)=\psi(x,T,\Theta,\xi)$
are the solutions to the system
\[\begin{cases}
\psi'_{\xi_j}(X,\Xi)=\vp'_{j, \xi}(x_{j-1},\xi_{j})-x_j=0 & j=1,\ldots,M,
\\
\psi'_{x_j}(X,\Xi)=\vp'_{j+1, x}(x_j,\xi_{j+1})-\xi_j=0 & j=1,\ldots,M,
\end{cases}
\]
in the unknowns $(T,\Theta)$. That is $(Y,N)=(Y_1,\ldots,Y_M,N_1,\ldots,N_M)(x,\xi)$ satisfies, if $M=1$,
\beqs\label{(C1)}\begin{cases}
Y_1=\vp'_{1, \xi}(x,N_{1}) 
\\
N_1=\vp'_{2, x}(Y_1,\xi),
\end{cases}
\eeqs
or, if $M\ge2$,
\beqs\label{(C)}\begin{cases}
Y_1=\vp'_{1,\xi}(x,N_{1}) 
\\
Y_j=\vp'_{j,\xi}(Y_{j-1},N_{j}), & j=2,\ldots,M
\\
N_j=\vp'_{j+1, x}(Y_j,N_{j+1}), & j=1,\ldots,M-1
\\
N_M=\vp'_{M+1, x}(Y_M,\xi) .
\end{cases}
\eeqs
In the sequel we will only refer to the system \eqref{(C)}, tacitly meaning \eqref{(C1)} when $M=1$. Definition
\ref{mprod} below of the multi product of SG phase functions is analogous to the one given in \cite{Kumano-go:1}
for (local) symbols of H\"ormander type.
\begin{definition}[Multi-product of SG phase functions]\label{mprod}
If, for every fixed $(x,\xi)\in \R^{2n}$, the system \eqref{(C)} admits a unique solution $(Y,N)=(Y,N)(x,\xi)$, we define
\beqs\label{3.4}
\phi(x,\xi)=(\vp_1\ \sharp\ \cdots\ \sharp\ \vp_{M+1})(x,\xi):=\psi(x,Y(x,\xi),N(x,\xi),\xi).
\eeqs 
The function $\phi$ is called multi-product of the SG phase functions $\vp_1,\ldots,\vp_{M+1}.$
\end{definition}

\begin{example}\label{32}
The simplest case of a well-defined multi-product of  SG phase functions is given by the sharp product $\varphi\ \sharp\ \varphi_0$, where $\varphi\in\mathcal P_r$ and $\varphi_0(x,\xi)=x\cdot\xi.$
Indeed, the critical points $(Y,N)$ of the function 
\[
	\widetilde{\psi}(x_1,\xi_1)=\psi(x,x_1,\xi_1,\xi)=\varphi(x,\xi_1)-x_1\cdot \xi_1+x_1\cdot \xi
\]
 are given by $(Y,N)(x,\xi)=(\varphi'_\xi(x,\xi), \xi)$. The multi-product $\varphi\ \sharp\ \varphi_0$ is so defined by
\[
	\phi(x,\xi)=\psi(x,\varphi'_\xi(x,\xi),\xi, \xi)=\varphi(x,\xi)-\varphi'_\xi(x,\xi)(\xi-\xi)=\varphi(x,\xi).
\]
Similarly, the multi-product $\varphi_0\ \sharp\ \varphi$ is well defined. Indeed, the function 
\[
	\widetilde{\psi}(x_1,\xi_1)=\psi(x,x_1,\xi_1,\xi)=x\cdot\xi_1-x_1\cdot\xi_1+\varphi(x_1,\xi)
\]
has critical points $(Y,N)(x,\xi)=(x, \varphi'_x(x,\xi))$, and
\[
	\phi(x,\xi)=\psi(x, x, \varphi'_x(x,\xi),\xi)=(x-x)\cdot\varphi'_x(x,\xi)+\varphi(x,\xi)=\varphi(x,\xi).
\]
Notice that we have proved here above that for every $\varphi\in\mathcal P_r$ the identity $$\varphi\ \sharp\ \varphi_0=\varphi_0\ \sharp\ \varphi=\varphi$$ holds true. That is,  
the multi-product of SG phase functions defined in \eqref{3.4} admits the trivial phase function $\varphi_0(x,\xi)=x\cdot\xi$ as identity element.
\end{example}
\begin{example} 
A situation where \eqref{3.4} is well defined, which is interesting for applications, see Section \ref{sec:fundsol}, is given by the 
multi-product of solutions to the eikonal equation \eqref{eik} on different, neighboring time intervals.
Indeed, the critical points $(Y,N)(x,\xi)$ of the function 
\[
	\widetilde{\psi}_{tsr}(x_1,\xi_1):=\psi_{tsr}(x,x_1,\xi_1,\xi)=\varphi(t,s,x,\xi_1)-x_1\cdot\xi_1+\varphi(s,r,x_1,\xi)
\]
are given by
\beqs\label{cr}
\begin{cases}
\psi'_{rst,x_1}(x,x_1,\xi_1,\xi)=-\xi_1+\varphi'_{x}(s,r,x_1,\xi)=0
\\
\psi'_{rst,\xi_1}(x,x_1,\xi_1,\xi)=\varphi'_{\xi}(t,s,x,\xi_1)-x_1=0.
\end{cases}\eeqs
The Jacobian matrix with respect to $(x_1,\xi_1)$ of the system \eqref{cr} is
\[
J(t,s,r,x,x_1,\xi_1,\xi)=\left(
\begin{array}{cc}
\varphi''_{xx}(s,r,x_1,\xi) &-I
\\
-I& \varphi''_{\xi\xi}(t,s,x,\xi_1)
\end{array}
\right),
\] 
where $I$ is the $(n\times n)$-dimensional unit matrix.
By \eqref{eik}, $\det J(t,r,r,x,x_1,\xi_1,\xi)=1$. Thus, taking a small interval $[0,T_0]$ such that $\det J(t,s,r,x,x_1,\xi_1,\xi)>0$ for all $r,s,t$ such that $0\leq r\leq s\leq t\leq T_0$ and all $(X,\Xi)\in\R^{4n}$, by the implicit function theorem it follows that the system \eqref{cr} admits a unique solution $(Y,N)_{tsr}=(Y_{tsr},N_{tsr})(x,\xi)=(Y(t,s,r,x,\xi),N(t,s,r,x,\xi))$. 
The multi-product
\begin{align*}
\phi_{tsr}(x,\xi)&=\phi(t,s,r,x,\xi)=(\varphi_{ts}\ \sharp\ \varphi_{sr})(x,\xi)=\psi_{tsr}(x,Y_{tsr}(x,\xi),N_{tsr}(x,\xi),\xi)
\\
&=\varphi(t,s,x,N_{tsr}(x,\xi))-Y_{tsr}(x,\xi)\cdot N_{tsr}(x,\xi)+\varphi(s,r,Y_{tsr}(x,\xi),\xi)
\end{align*}
is then well defined. Moreover, it is quite simple to show, in view of to Proposition \ref{trovala!}, that the multi-product $\varphi_{ts}\ \sharp\ \varphi_{sr}$ satisfies the associative law
\beqs\label{ass}
\varphi_{ts}\ \sharp\ \varphi_{sr}=\varphi_{tr}, \quad 0\leq r\leq s\leq t\leq T_0.
\eeqs
Indeed, $\phi(t,s,r, x,\xi)$ does not depend on $s$:
\begin{align*}
\frac{d}{ds}[\phi(t,s,r,x,\xi)]\!&=\!(\partial_s\varphi)(t,s,x,N_{tsr}(x,\xi))\!+\!
\varphi'_\xi(t,s,x,N_{tsr}(x,\xi))\!\cdot\!(\partial_s N)(t,s,r,x,\xi)
\\
&-(\partial_s Y)(t,s,r,x,\xi)\cdot N(t,s,r,x,\xi)-Y(t,s,r,x,\xi)\cdot(\partial_s N)(t,s,r,x,\xi)
\\
&+(\partial_t\varphi)(s,r,Y_{tsr}(x,\xi),\xi)+\varphi'_x(s,r,Y_{tsr}(x,\xi),\xi)\cdot(\partial_s Y)(t,s,r,x,\xi)=0,
\end{align*}
since, by \eqref{eik}, \eqref{eiks} and the definition \eqref{cr} of the critical point $(Y,N)_{tsr}$, we have
\begin{align*}
	\varphi'_x(s,r,Y_{tsr}(x,\xi),\xi)&=N(t,s,r,x,\xi),
	\\
	\varphi'_\xi(t,s,x,N_{tsr}(x,\xi))&=Y(t,s,r,x,\xi),
	\\
	(\partial_t\varphi)(s,r,Y_{tsr}(x,\xi),\xi)&=\phantom{-}a(s,Y_{tsr}(x,\xi),\varphi^\prime_x(s,r,Y_{tsr}(x,\xi),\xi))
	\\
	&=\phantom{-}a(s,Y_{tsr}(x,\xi),N_{tsr}(x,\xi)),
	\\
	(\partial_s\varphi)(t,s,x,N_{tsr}(x,\xi))&=-a(s,\varphi^\prime_\xi(t,s,x,N_{tsr}(x,\xi)),N_{tsr}(x,\xi))
	\\
	&=-a(s,Y_{tsr}(x,\xi),N_{tsr}(x,xi)).
\end{align*}
This gives, with $\varphi_0(x,\xi)=x\cdot\xi$,
\begin{align*}
(\varphi_{ts}\ \sharp\ \varphi_{sr})(x,\xi)&=\phi(t,s,r,x,\xi)
=\phi(t,r,r,x,\xi)=(\varphi_{tr}\ \sharp\ \varphi_{rr})(x,\xi)=(\varphi_{tr}\ \sharp\ \varphi_0)(x,\xi)
\\
&= \varphi_{tr}(x,\xi),
\end{align*}
by Example \ref{32}, as claimed.
\end{example}

Now we want to show that under assumption \eqref{1/4} the multi-product $\phi(x,\xi)$ of Definition \ref{mprod} is well defined on $\R^{2n}$, and it is a regular SG phase function itself. 
To this aim, we switch from the system \eqref{(C)} in the unknown $(Y,N)$ to the equivalent system \eqref{(C')} in the unknown 
$(\bar Y,\bar N)=(y_1,\ldots,y_M, \eta_1,\ldots,\eta_M)\in\R^{2Mn}$ as follows. Define 
\beqs\label{zj}
\begin{cases}
z^0:=0
\\
z^j:=\sum_{k=1}^j y_k\,, & j=1,\ldots,M
\\
\zeta^j:=\sum_{k=j}^M \eta_k\,, & j=1,\ldots,M
\\
\zeta^{M+1}:=0,
\end{cases}
\eeqs
and then consider the system 
\beqs\label{(C')}\begin{cases}
y_k=J'_{k, \xi}(x+z^{k-1},\xi+\zeta^{k}), & k=1,\ldots,M
\\
\eta_k=J'_{k+1, x}(x+z^k,\xi+\zeta^{j+1}), & k=1,\ldots,M.
\end{cases}
\eeqs
We have that:
\begin{lemma}
For every fixed $(x,\xi)\in\R^{2n}$, $(Y,N)(x,\xi)$ is a solution of \eqref{(C)} if and only if 
$(\bar Y,\bar N)(x,\xi)=(y_1,\ldots,y_M,\eta_1,\ldots,\eta_m)(x,\xi)$, defined by 
\beqs\label{star}
\begin{cases}
y_1=Y_1-x
\\
y_j=Y_j-Y_{j-1} & j=2,\ldots,M
\\
\eta_j=N_j-N_{j+1} & j=1,\ldots,M-1
\\
\eta_M=N_M-\xi,
\end{cases}
\eeqs
is a solution of \eqref{(C')}.
\end{lemma}
\begin{proof}
Substituting \eqref{star} in \eqref{zj}, we immediately get the relation
\beqs\label{simplificationH}\begin{cases}
Y_j=x+z^j
\\
N_j=\xi+\zeta^j.
\end{cases}\eeqs
By this, it follows that $(Y,N)$ is a solution of \eqref{(C)} if and only if 
\[
\begin{cases}
x+z^j=\vp'_{j, \xi}(x+z^{j-1},\xi+\zeta^{j}) & j=1,\ldots,M
\\
\xi+\zeta^j=\vp'_{j+1, x}(x+z^j,\xi+\zeta^{j+1}) & j=1,\ldots,M;
\end{cases}
\]
by substituting $\vp_j(x,\xi)=J_j(x,\xi)+x\cdot\xi$ we obtain
\[
\begin{cases}
z^j-z^{j-1}=J'_{j, \xi}(x+z^{j-1},\xi+\zeta^{j}) & j=1,\ldots,M
\\
\zeta^j-\zeta^{j+1}=J'_{j+1, x}(x+z^j,\xi+\zeta^{j+1}) & j=1,\ldots,M,
\end{cases}
\]
which is exactly \eqref{(C')}, in view of \eqref{zj}.
\end{proof}
\noindent
We are then reduced to prove the following Theorem \ref{exists}.
\begin{theorem}\label{exists}
Under the assumption \eqref{1/4}, for every fixed $(x,\xi)\in\R^{2n}$ there exists a unique solution $(\bar Y, \bar N)(x,\xi)$ of \eqref{(C')}. Moreover, the solution $(\bar Y, \bar N)$ satisfies
\beqs\label{yeta} |y_k|\leq\frac 43\tau_k\x,\qquad |\eta_k|\leq \frac43\tau_{k+1 }\<\xi\>,\quad k=1,\ldots,M,
\eeqs
and the functions $z_j$ and $\zeta_j$ in \eqref{zj} satisfy
\beqs\label{zzeta} |z^j|\leq \frac13\x,\qquad |\zeta^j|\leq \frac13\<\xi\>, \quad j=1,\ldots,M.
\eeqs
\end{theorem}
\begin{remark}\label{booh}
We aim at obtaining a  solution $(Y,N)$ such that $\phi=\psi(.,Y,N,..)\in \mathcal P_r(\tau)$. 
By Definition \ref{mprod}, recalling that $\psi$ a smooth function, it is enough to show
that $(Y,N)$ is of class $C^\infty(\R^{2n})$, that $Y_j\in S^{1,0}$, $N_j\in S^{0,1}$, and that 
$\<Y_j(x,\xi)\>\asymp\x$ as $|x|\to\infty$,  $\<N_j(x,\xi)\>\asymp\<\xi\>$ as $|\xi|\to\infty$. To get these last equivalences, it is sufficient to prove the existence of a constant $k\in(0,1)$ such that $|Y_j(x,\xi)-x|\leq k\x$ and $|N_j(x,\xi)-\xi|\leq k\<\xi\>$. Indeed, the following implication holds:
\begin{equation}\label{palla}
|b|\leq k\<a\>,\; k\in (0,1),\ a,b\in\R^n \Longrightarrow (1-k)\<a\>\leq \<a+b\>\leq (1+k)\<a\>.
\end{equation}
Formula \eqref{zzeta} gives precisely the desired estimates, with $k=1/3$, owing to \eqref{simplificationH}. Theorem \ref{exists} then ensures that the multi-product is well-defined. We show that 
$(Y,N)\in C^\infty(\R^{2n})$ in the subsequent Theorem \ref{cinfty}.
\end{remark}

\begin{proof}[Proof of Theorem \ref{exists}]
We divide the proof into two steps. In step one we suppose the existence of a solution $(\bar Y, \bar N)$ of \eqref{(C')} and prove that such solution satisfies \eqref{yeta} and that \eqref{zzeta} holds. In step two we show, by a fixed point argument, the existence and uniqueness of the solution $(\bar Y, \bar N)$.
\\
\emph{Step 1.} If $(\bar Y, \bar N)$ is a solution of \eqref{(C')}, then by  \eqref{(C')} and \eqref{hyp} we get, 
for any $(x,\xi)\in\R^{2n}$,
\[\begin{cases}
|y_k|\leq \tau_k\<x+z^{k-1}\>
\\
|\eta_k|\leq \tau_{k+1}\<\xi+\zeta^{k+1}\>
\end{cases}\]
for $k=1,\ldots,M$. Now, using the inequality
\begin{equation}\label{stellina}
\<x+y\>\leq \x+|y|\,\;\forall x,y\in\R^n 
\end{equation}
and definition \eqref{zj}, we get, for $k=1,\ldots,M$ and any $(x,\xi)\in\R^{2n}$,
\beqs\label{ee}\begin{cases}
|y_k|\leq \tau_k\left(\x+|z^{k-1}|\right)\leq \tau_k\left(\x+\ds\sum_{j=1}^M|y_j|\right),
\\
|\eta_k|\leq \tau_{k+1}\left(\<\xi\>+|\zeta^{k+1}|\right)\leq \tau_{k+1}\left(\<\xi\>+\ds\sum_{j=1}^M|\eta_j|\right),
\end{cases}
\eeqs
so that
\beqs\label{eee}
\begin{cases}
\ds\sum_{k=1}^M|y_k|\leq \ds\sum_{k=1}^M\tau_k\left(\x+\ds\sum_{k=1}^M|y_k|\right)=:\bar \tau_M\left(\x+\ds\sum_{k=1}^M|y_k|\right),
\\
\ds\sum_{k=1}^M|\eta_k|\leq \ds\sum_{k=1}^M\tau_{k+1}\left(\<\xi\>+\ds\sum_{k=1}^M|\eta_k|\right)=:\bar\tau_{M+1}\left(\<\xi\>+\ds\sum_{k=1}^M|\eta_k|\right).
\end{cases}
\eeqs
The two inequalities here above are of the form $\alpha\leq \tau(\x+\alpha)$ with $\tau< \tau_0<1/4$ by assumption \eqref{1/4}, so they give 
\[\alpha\leq \frac{\tau}{1-\tau}\x<\frac13\x, \]
and, coming back to \eqref{eee}, we have, for any $(x,\xi)\in\R^{2n}$,
\[
|z^k|\leq\ds\sum_{j=1}^k|y_j|< \frac13\x, \qquad
|\zeta^k|\leq \ds\sum_{j=k}^M|\eta_j|< \frac13\<\xi\>,
\]
that is \eqref{zzeta}. Substituting in \eqref{ee} we obtain 
\[|y_k|\leq \tau_k\left(\x+\frac13\x\right)=\frac43\tau_k\x,\qquad |\eta_k|\leq \tau_{k+1}\left(\<\xi\>+\frac13\<\xi\>\right)=\frac43\tau_{k+1}\<\xi\>,\]
that is \eqref{zzeta}.
\\
\emph{Step 2.} Since we have shown that every solution $(\bar Y, \bar N)$ of \eqref{(C')} satisfies \eqref{zzeta} for
any $(x,\xi)\in\R^{2n}$, to show existence and uniqueness of a solution to \eqref{(C')} in $\R^{2Mn}$ it is sufficient to show existence and uniqueness of $(\bar Y, \bar N)$ in the space
\[
\Sigma=\Sigma_{x,\xi}
:=\left\{(y_1,\ldots,y_M,\eta_1,\ldots,\eta_M)\in\R^{2Mn}:\  \ds\sum_{k=1}^M|y_k\vert\leq \frac13\x,\ \ds\sum_{k=1}^M|\eta_k|\leq \frac13\<\xi\>\right\},
\]
$(x,\xi)\in\R^{2n}$, which is a metric space with norm 
\[\|(y_1,\ldots,y_M,\eta_1,\ldots,\eta_M)\|_\Sigma:=\sum_{k=1}^M\left(\x^{-1}|y_k|+\<\xi\>^{-1}|\eta_k|\right).\]
We define the map
\[T=T_{x,\xi}:\Sigma\longrightarrow \Sigma\]
by $T(y_1,\ldots,y_M,\eta_1,\ldots,\eta_M):=(w_1,\ldots,w_M,\omega_1,\ldots,\omega_M)$, where, for $ k=1,\ldots, M$,
$(x,\xi)\in\R^{2n}$,
\[\begin{cases}
w_k=J'_{k, \xi}(x+z^{k-1},\xi+\zeta^{k}) 
\\
\omega_k=J'_{k+1, x}(x+z^k,\xi+\zeta^{k+1}).
\end{cases}
\]
The map $T$ is well defined. Indeed, by \eqref{hyp}, \eqref{stellina} and \eqref{zzeta} we have, for any $(x,\xi)\in\R^{2n}$,
\beqs
\begin{cases}
|w_k|\leq \tau_k\<x+z^{k-1}\>\leq \tau_k(\x+\frac13\x)=\frac43\tau_k\x
\\
|\omega_k|\leq \tau_{k+1}\<\xi+\zeta^{k+1}\>\leq \tau_{k+1}(\<\xi\>+\frac13\<\xi\>)=\frac43\tau_{k+1}\<\xi\>,
\end{cases}
\eeqs
so that
$$\ds\sum_{k=1}^M|w_k|\leq \frac43\x\cdot\ds\sum_{k=1}^M\tau_k<\frac13\x,\ \quad{\rm and}\quad\ \ds\sum_{k=1}^M|\omega_k|\leq \frac43\<\xi\>\cdot\ds\sum_{k=1}^M\tau_{k+1}<\frac13\<\xi\>.$$
By \eqref{(C')}, to show existence and uniqueness of $(\bar Y, \bar N)=(\bar Y, \bar N)(x,\xi)$ is equivalent to show existence and uniqueness of a fixed point $(\bar Y, \bar N)$ of the map $T$. We show here below that, under assumption \eqref{1/4}, $T$ is a contraction on $\Sigma$, so it admits a unique fixed point $(\bar Y,\bar N).$
\\

Let us consider two arbitrary points 
\[
(Y,N)=(y_1,\ldots,y_M,\eta_1,\ldots,\eta_M), 
\;
(\widetilde Y,\widetilde N)=(\widetilde y_1,\ldots,\widetilde y_M,\widetilde\eta_1,\ldots,\widetilde\eta_M)\in \Sigma,
\]
and let 
\[
T(Y,N)=(w_1,\ldots,w_M,\omega_1,\ldots,\omega_M), \;
T(\widetilde Y,\widetilde N)=(\widetilde w_1,\ldots,\widetilde w_M,\widetilde \omega_1,\ldots,\widetilde \omega_M).
\]
For every fixed $k=1,\ldots,M$, $(x,\xi)\in\R^{2n}$, we have
\beqsn
\widetilde w_k-w_k&=&J'_{k, \xi}(x+\widetilde z^{k-1},\xi+\widetilde \zeta^{k}) -J'_{k, \xi}(x+z^{k-1},\xi+\zeta^{k}) 
\\
&=&(\widetilde z^{k-1}-z^{k-1})\ds\int_0^1J''_{k, \xi x}(x+z^{k-1}+\theta(\widetilde z^{k-1}-z^{k-1}),\xi+\zeta^{k})d\theta
\\
&+&
 (\widetilde\zeta^{k}-\zeta^{k})\ds\int_0^1J''_{k, \xi \xi}(x+z^{k-1},\xi+\zeta^{k}+\theta(\widetilde\zeta^{k}-\zeta^{k}))d\theta
\eeqsn
and from \eqref{hyp} we get
\beqsn
|\widetilde w_k-w_k|&\leq &\tau_k\left(|\widetilde z^{k-1}-z^{k-1}|+|\widetilde\zeta^{k}-\zeta^{k}|\<x+z^{k-1}\>\ds\int_0^1\!\!\!\<\xi+\zeta^{k}+\theta(\widetilde\zeta^{k}-\zeta^{k})\>^{-1}d\theta\right).
\eeqsn
By inequality \eqref{palla} with $b=z^k$ and $k=1/3$ we get $\frac23\x\leq\< x+z^k\>\leq\frac43\x$; the same inequality with $b=\zeta^k+\theta(\widetilde\zeta^k-\zeta^k)$ and $k=1/3$ gives $\frac23\<\xi\>\leq\< \xi+\zeta^k+\theta(\widetilde\zeta^k-\zeta^k)\>\leq\frac43\<\xi\>$; substituting these inequalities into the estimate of $|\widetilde w_k-w_k|$ we come to
\beqsn
|\widetilde w_k-w_k|&\leq &\tau_k\left(|\widetilde z^{k-1}-z^{k-1}|+|\widetilde\zeta^{k}-\zeta^{k}|2\x\ds\int_0^1\!\!\!\<\xi\>^{-1}d\theta\right)
\\
&\leq& \tau_k\ds\sum_{j=1}^M\left(|\widetilde y_j-y_j|+|\widetilde\eta_{j}-\eta_{j}|2\x\<\xi\>^{-1}\right).
\eeqsn
Similarly:
\beqsn
|\widetilde \omega_k-\omega_k|&\leq &|\widetilde z^{k}-z^{k}|\left\vert\ds\int_0^1J''_{k+1, xx}(x+z^{k}+\theta(\widetilde z^{k}-z^{k}),\xi+\zeta^{k+1})d\theta\right\vert
\\
&+&
|\widetilde\zeta^{k+1}-\zeta^{k+1}|\left\vert\ds\int_0^1J''_{k+1, x,\xi}(x+z^{k},\xi+\zeta^{k+1}+\theta(\widetilde\zeta^{k+1}-\zeta^{k+1}))d\theta\right\vert
\\
&\leq & \tau_{k+1}\left(|\widetilde z^{k}-z^{k}|2\x^{-1}\<\xi\>+|\widetilde\zeta^{k+1}-\zeta^{k+1}|\right)
\\
&\leq & \tau_{k+1}\ds\sum_{j=1}^M\left(|\widetilde y_{j}-y_{j}|2\x^{-1}\<\xi\>+|\widetilde\eta_{k+1}-\eta_{k+1}|\right).
\eeqsn
Thus
\beqsn
\|T(Y,N)-T(\widetilde Y,\widetilde N)\|_\Sigma&=&\sum_{k=1}^M\left(\x^{-1}|\widetilde w_k-w_k|+\<\xi\>^{-1}|\widetilde\omega_k-\omega_k|\right)
\\
&\leq &\sum_{k=1}^M\left(
\tau_k\ds\sum_{j=1}^M\left(\x^{-1}|\widetilde y_j-y_j|+2\<\xi\>^{-1}|\widetilde\eta_{j}-\eta_{j}|\right)\right.
\\
&&\left.+\tau_{k+1}\ds\sum_{j=1}^M\left(|\widetilde y_{j}-y_{j}|2\x^{-1}+|\widetilde\eta_{j}-\eta_{j}|\<\xi\>^{-1}\right)
\right)
\\
&\leq &\ds\sum_{k=1}^M\max\{\tau_k,\tau_{k+1}\}3\ds\sum_{j=1}^M\left(|\widetilde y_{j}-y_{j}|\x^{-1}+|\widetilde\eta_{j}-\eta_{j}|\<\xi\>^{-1}\right)
\\
&\leq & 3\tau_0\|(Y,N)-(\widetilde Y,\widetilde N)\|_\Sigma.
\eeqsn
This shows that the map $T$ is Lipschitz continuous, with Lispchitz constant $3\tau_0<1$. It follows that $T$ is a strict 
contraction on $\Sigma$, which then admits a unique fixed point $(\bar Y, \bar N)\in\Sigma$, 
for any $(x,\xi)\in\R^{2n}$. Such fixed point obviously gives the unique solution of \eqref{(C')}. The proof is complete.
\end{proof}
\begin{theorem}\label{cinfty}
The unique solution $(\bar Y, \bar N)=(\bar Y, \bar N)(x,\xi)$ of $\eqref{(C')}$ is of class $C^\infty(\R^{2n}).$
\end{theorem}
\begin{proof}
For $(Y,N)\in \R^{2Mn}$ and $(x,\xi)\in \R^{2n}$, we define the function $$F(Y,N;x,\xi):=(F_1,\ldots, F_M)(Y,N;x,\xi),$$ with values in $\R^{2M},$ where for all $k=1,\ldots,M$,
\[F_k(Y,N;x,\xi):=\left(y_k-J'_{k, \xi}(x+z^{k-1},\xi+\zeta^{k}),\ \eta_k-J'_{k+1, x}(x+z^k,\xi+\zeta^{k+1})\right). \]
We apply the implicit function Theorem to the function $F$, which is clearly of class $C^\infty$ with respect to all variables, being $J_k$ a $C^\infty$ function for all $k=1,\ldots,M$. For every fixed $(x,\xi)$ we have that
\[F((\bar Y,\bar N)(x,\xi);x,\xi)=0,\]
since $(\bar Y,\bar N)$ is the solution of \eqref{(C')}. Moreover, we are going to prove here below that 
\beqs\label{det}
\det\left( \frac{\partial F}{\partial(Y,N)}((\bar Y,\bar N)(x,\xi);x,\xi)\right)\neq 0.
\eeqs
This means that the implicitly defined function $(\bar Y,\bar N)(x,\xi)$ has the same regularity as $F$, so it is of class $C^\infty(\R^{2n})$. To complete the proof, it remains only to show that \eqref{det} holds true. 
\\
\noindent Let us compute the entries of the $2M\times 2M$ matrix $ \frac{\partial F}{\partial(Y,N)}(Y,N;x,\xi)$. For every fixed $k=1,\ldots,M$, $(x,\xi)\in\R^{2n}$, we have
\beqsn
F'_{k, y_j}\!(Y,N;x,\xi)\!\!=\!\!\begin{cases}
\!\left(\!-J''_{k, \xi x}(x+z^{k-1},\xi+\zeta^{k}),\ -J''_{k+1, xx}(x+z^k,\xi+\zeta^{k+1})\!\right)\!, &\hspace{-3mm} 1\leq j\leq k-1
\\
\!\left(\!1,\ -J''_{k+1, xx}(x+z^k,\xi+\zeta^{k+1})\!\right)\!, &\hspace{-3mm} j=k
\\
\!(0,\ 0), &\hspace{-3mm} k+1\leq j\leq M,
\end{cases}
\eeqsn
and
\beqsn
F'_{k, \eta_j}\!(Y,N;x,\xi)\!\!=\!\!\begin{cases}
\!(0,\ 0), & \hspace{-3mm} 1\leq j\leq k-1
\\
\!\left(\!-J''_{k, \xi\xi}(x+z^{k-1},\xi+\zeta^{k}),\ 1\!\right)\!, & \hspace{-3mm} j=k
\\
\!\left(\!-J''_{k, \xi\xi}(x+z^{k-1},\xi+\zeta^{k}),\ -J''_{k+1, x\xi}(x+z^k,\xi+\zeta^{k+1})\!\right)\!, &\hspace{-3mm}  k+1\leq j\leq M,
\end{cases}
\eeqsn
so we can write
\[ \frac{\partial F}{\partial(Y,N)}(Y, N;x,\xi)=\left(
\begin{array}{cc}
I-H_{11}(Y,N;x,\xi)& -H_{12}(Y,N;x,\xi)
\\
-H_{21}(Y,N;x,\xi) & I-H_{22}(Y,N;x,\xi)
\end{array}
\right),\]
where $I$ stands for the identity $M\times M$ matrix, and 
\[H_{1,1}=\left(
\begin{array}{cccc}
0&0&\cdots&0
\\
J''_{2, \xi x}&0&\ddots&0
\\
\cdots&\cdots&\ddots&\vdots
\\
J''_{M, \xi x}&\cdots&J''_{M, \xi x}& 0
\end{array}
\right),\quad
H_{1,2}=\left(
\begin{array}{cccc}
J''_{1, \xi\xi}&\cdots&\cdots&J''_{1, \xi\xi}
\\
0&J''_{2, \xi\xi}&\cdots&J''_{2, \xi\xi}
\\
\vdots&\ddots&\ddots&\vdots
\\
0&\cdots&0&J''_{M, \xi \xi}\end{array}
\right)\]
\[H_{2,1}=\left(
\begin{array}{cccc}
J''_{2, xx}&0&\cdots&0
\\
J''_{3, xx}&J''_{3, xx}&\cdots&0
\\
\vdots&\cdots&\ddots&\vdots
\\
J''_{M+1, xx}&\cdots&\cdots&J''_{M+1, xx}\end{array}
\right),\quad
H_{2,2}=\left(
\begin{array}{cccc}
0&J''_{2, x\xi}&\cdots&J''_{2, x\xi}
\\
0&0&\cdots&\vdots
\\
\vdots&\vdots&\ddots&J''_{M, x\xi}
\\
0&\cdots&\cdots&0\end{array}
\right).\]
Let us estimate the matrix norm of each one of the $H_{ij}$:
\beqsn
\| H_{11}(Y,N;x,\xi)\|&=&\max_{j=1,\ldots,M}\ds\sum_{i=1}^M |(h_{11})_{ij}|\leq \max_{j=1,\ldots,M}\sum_{i=j+1}^M\tau_i\leq\ds\sum_{j=1}^M\tau_j
\\
\| H_{12}(Y,N;x,\xi)\|&=&\max_{j=1,\ldots,M}\ds\sum_{i=1}^M |(h_{12})_{ij}|\leq \max_{j=1,\ldots,M}\sum_{i=1}^j\tau_i\<x+z^{i-1}\>\<\xi+\zeta^i\>^{-1}
\\
\| H_{21}(Y,N;x,\xi)\|&=&\max_{j=1,\ldots,M}\ds\sum_{i=1}^M |(h_{21})_{ij}|\leq \max_{j=1,\ldots,M}\sum_{i=j}^M\tau_{i+1}\<x+z^{i}\>^{-1}\<\xi+\zeta^{i+1}\>
\\
\| H_{22}(Y,N;x,\xi)\|&=&\max_{j=1,\ldots,M}\ds\sum_{i=1}^M |(h_{22})_{ij}|\leq \max_{j=1,\ldots,M}\sum_{i=1}^{j-1}\tau_{i+1}\leq\ds\sum_{j=1}^M\tau_j.
\eeqsn
With the choice $(Y,N)=(\bar Y,\bar N)(x,\xi)$ these estimates become, via formula \eqref{simplificationH} and Remark \ref{booh},
\beqsn
\| H_{11}((\bar Y,\bar N)(x,\xi);x,\xi)\|\leq \ds\sum_{j=1}^M\tau_j,\qquad
\| H_{12}((\bar Y,\bar N)(x,\xi);x,\xi)\|\leq 2\x\<\xi\>^{-1}\sum_{i=1}^M\tau_i,
\\
\| H_{21}((\bar Y,\bar N)(x,\xi);x,\xi)\|\leq2\x^{-1}\<\xi\>\ds\sum_{i=1}^M \tau_{i},
\qquad
\| H_{22}((\bar Y,\bar N)(x,\xi);x,\xi)\|\leq\ds\sum_{j=1}^M\tau_j.
\eeqsn
Now, since $\det(I-H_{11})=1$, being $H_{11}$ triangular with null diagonal, 
we have
\begin{align*} \det&\frac{\partial F}{\partial(Y,N)}((\bar Y,\bar N)(x,\xi);x,\xi)
\\
=&\det \left(
\begin{array}{cc}
I-H_{11}((\bar Y,\bar N)(x,\xi);x,\xi)& -\<\xi\>\x^{-1}H_{12}((\bar Y,\bar N)(x,\xi);x,\xi)
\\
-\x \<\xi\>^{-1}H_{21}((\bar Y,\bar N)(x,\xi);x,\xi) & I-H_{22}((\bar Y,\bar N)(x,\xi);x,\xi)
\end{array}
\right)
\\
=&\det\left(I-\left(
\begin{array}{cc}
H_{11}((\bar Y,\bar N)(x,\xi);x,\xi)& \<\xi\>\x^{-1}H_{12}((\bar Y,\bar N)(x,\xi);x,\xi)
\\
\x \<\xi\>^{-1}H_{21}((\bar Y,\bar N)(x,\xi);x,\xi) & H_{22}((\bar Y,\bar N)(x,\xi);x,\xi)
\end{array}
\right)\right)
\\
=&\det(I-A(x,\xi)),
\end{align*}
with 
\begin{align*}
\|A(x,\xi)\|&=\max\{\|H_{11}((\bar Y,\bar N)(x,\xi);x,\xi)\|+\|\x \<\xi\>^{-1}H_{21}((\bar Y,\bar N)(x,\xi);x,\xi)\|, 
\\
&\phantom{=\max\{\;}\| H_{22}((\bar Y,\bar N)(x,\xi);x,\xi)\|+\|\<\xi\>\x^{-1}H_{12}((\bar Y,\bar N)(x,\xi);x,\xi)\|\}
\\
&\leq 3\ds\sum_{j=1}^M\tau_j\leq 3\tau_0<\ds\frac34,
\end{align*}
and applying Proposition \ref{5.3} below, cfr. \cite{Kumano-go:1}, we get 
$\det(I-A(x,\xi))\geq 4^{-2M}>0$. That is, \eqref{det} holds true, and the proof is complete.
\end{proof}
\begin{proposition}[Proposition 5.3, page 336 in \cite{Kumano-go:1}]\label{5.3}
Let $A=(a_{ij})_{1\leq i,j\leq \ell}$ be a real matrix and suppose that there exists a constant $c_0\in [0,1)$ such that 
\[\|A\|:=\max_{j=1,\ldots,\ell}\ds\sum_{i=1}^\ell |a_{ij}|\leq c_0.\]
Then,
\[
	(1-c_0)^\ell\leq \det (I-A)\leq (1+c_0)^\ell.
\]
\end{proposition}
The following theorem gives crucial estimates of the unique $C^\infty$ solution $(Y,N)$ of \eqref{(C)}. 
\begin{theorem}\label{stime}
Under the assumptions \eqref{1/4} and \eqref{hyp}, 
the unique $C^\infty$ solution $(Y,N)(x,\xi)$ of \eqref{(C)} satisfies:
\begin{align}
\label{stimaY} 
|\partial_\xi^\alpha\partial_x^\beta(Y_j-Y_{j-1})(x,\xi)|&\leq c_{\alpha,\beta}\tau_j \<\xi\>^{-|\alpha|}\x^{1-|\beta|},
\\
\label{stimaN}
|\partial_\xi^\alpha\partial_x^\beta(N_j-N_{j+1})(x,\xi)|&\leq c_{\alpha,\beta}\tau_{j+1} \<\xi\>^{1-|\alpha|}\x^{-|\beta|},
\end{align}
for all $\alpha, \beta \in\Z^n_+,\; j=1,\ldots,M$, $x,\xi\in\R^n$,
 with constants $c_{\alpha,\beta}$ not depending on $j$ and $M$. Moreover, 
\beqs\label{324}
\{(Y_j-Y_{j-1})(x,\xi)/\tau_j\}_{j\geq 1}\ is\ bounded\ in\ S^{0,1},
\\
\{(N_j-N_{j+1})(x,\xi)/\tau_{j+1}\}_{j\geq 1}\ is\ bounded\ in\ S^{1,0}.
\eeqs
\end{theorem}
\begin{proof}
Estimates \eqref{stimaY}, \eqref{stimaN} in the case $\alpha=\beta=0$ have already been proved, see \eqref{yeta} and \eqref{star}.
To prove the same estimates for $|\alpha+\beta|\geq 1$,  it is sufficient, by \eqref{yeta}, \eqref{star} and \eqref{marr}, to show that the solution $(\bar Y,\bar N)(x,\xi)$ of \eqref{(C')} is such that
\beqs\label{stimay} 
|\partial_\xi^\alpha\partial_x^\beta y_k(x,\xi)|\leq c_{\alpha,\beta}\|J_k\|_{2,|\alpha+\beta|-1} \<\xi\>^{-|\alpha|}\x^{1-|\beta|},
\\
\label{stimaeta}
|\partial_\xi^\alpha\partial_x^\beta\eta_k(x,\xi)|\leq c_{\alpha,\beta}\|J_{k+1}\|_{2,|\alpha+\beta|-1} \<\xi\>^{1-|\alpha|}\x^{-|\beta|},
\eeqs
for $|\alpha+\beta|\geq 1,\; k=1,\ldots,M$, $x,\xi\in\R^{n}$. 
Estimates \eqref{stimay}, \eqref{stimaeta} are going to be proved by induction on $N=|\alpha+\beta|$.
\\
\emph{Step $N=1$.} We need to check \eqref{stimay}, \eqref{stimaeta} for the derivatives of order 1. Let us start with the derivatives with respect to $x$. By definition \eqref{(C')} of $y_k,\eta_k$, $k=1,\dots,M$, $x,\xi\in\R^n$, we have
\beqs\label{derx}
\hspace*{1.2cm}
\begin{cases}
y'_{k, x}=J''_{k, \xi x}(.+z^{k-1},..+\zeta^{k})(1+(z^{k-1})'_x)+J''_{k, \xi\xi}(.+z^{k-1},..+\zeta^{k})(\zeta^{k})'_x
\\
\eta'_{k, x}=J''_{k+1, xx}(.+z^k,..+\zeta^{k+1})(1+(z^k)'_x)+J''_{k+1, x\xi}(.+z^k,..+\zeta^{k+1})(\zeta^{k+1})'_x.
\end{cases}
\eeqs
By \eqref{hyp}, setting $h(x,\xi)=\norm{x}\norm{\xi}^{-1}$, we obtain
\beqsn
\|y'_{k, x}\|+h\cdot\|\eta'_{k, x}\|&\leq&\tau_k\left\{1+\|(z^{k-1})'_x\|+\<.+z^{k-1}\>\<..+\zeta^{k}\>^{-1}\|(\zeta^{k})'_x\|\right\}
\\
&+&\tau_{k+1}\cdot h\cdot\left\{\<.+z^k\>^{-1}\<..+\zeta^{k+1}\>(1+\|(z^k)'_x\|)+\|(\zeta^{k+1})'_x\|\right\};
\eeqsn
from \eqref{zzeta} we have $\frac23\x\leq \<x+z^{k-1}\>\leq\frac43\x$ and $\frac23\<\xi\>\leq \<\xi+\zeta^{k}\>\leq \frac43\<\xi\>$, so we come to
\beqsn
\|y'_{k, x}\|+h\cdot\|\eta'_{k, x}\|&\leq&\tau_k\left\{1+\|(z^{k-1})'_x\|+2\cdot h \cdot\|(\zeta^{k})'_x\|\right\}
\\
&+&\tau_{k+1}\left\{2+2\|(z^k)'_x\|+h\cdot\|(\zeta^{k+1})'_x\|\right\}
\\
&\leq &\tau_k\left\{1+\ds\sum_{k=1}^M\|y'_{k, x}\|+2\cdot h\cdot\ds\sum_{k=1}^M\|\eta'_{k, x}\|\right\}
\\
&+&\tau_{k+1}\left\{2+2\ds\sum_{k=1}^M\|y'_{k, x}\|+h\cdot\ds\sum_{k=1}^M\|\eta'_{k, x}\|\right\},
\eeqsn
where we have used also definition \eqref{zj}. Summing for $k=1,\ldots,M$, we get, for any $x,\xi\in\R^n$,
\beqsn
\ds\sum_{k=1}^M\left(\|y'_{k, x}\|+h\cdot\|\eta'_{k, x}\|\right)&\leq&\bar \tau_M\left\{1+\ds\sum_{k=1}^M\|y'_{k, x}\|+
2\cdot h\cdot\ds\sum_{k=1}^M\|\eta'_{k, x}\|\right\}
\\
&+&\bar\tau_{M+1}\left\{2+2\ds\sum_{k=1}^M\|y'_{k, x}\|+h\cdot\ds\sum_{k=1}^M\|\eta'_{k, x}\|\right\}
\\
&\leq &3\bar\tau_{M+1}\left\{1+\ds\sum_{k=1}^M\left(\|y'_{k, x}\|+h\cdot\|\eta'_{k, x}\|\right)\right\}.
\eeqsn
This last inequality immediately gives
\beqs\label{amount}
\ds\sum_{k=1}^M\left(\|y'_{k, x}\|+h\cdot\|\eta'_{k, x}\|\right)\leq\frac{3\tau_{M+1}}{1-3\tau_{M+1}}\leq\frac{3\tau_0}{1-3\tau_0}
\eeqs
with $1-3\tau_0>1-3/4=1/4>0$, so that the amount \eqref{amount} is finite (bounded by 3).
Coming back to \eqref{derx} and substituting there the estimate here above we get
\beqsn
\|y'_{k, x}\|&\leq& \|J_k\|_{2,0}\left\{1+\|(z^{k-1})'_x\|+2\cdot h\cdot\|\eta'_{k, x}\|\right\}
\\
&\leq &2\|J_k\|_{2,0}\left\{1+\ds\sum_{k=1}^M\left(\|y'_{k, x}\|+h\cdot\|\eta'_{k, x}\|\right)\right\}
\\
&\leq &2\|J_k\|_{2,0}\left(1+\frac{3\tau_0}{1-3\tau_0}\right)=:c_{0,1}\|J_k\|_{2,0},
\eeqsn
that is \eqref{stimay} with $\alpha=0$ and $|\beta|=1$. With similar computations we obtain
\beqsn
\|\eta'_{k, x}(x,\xi)\|&\leq& \|J_{k+1}\|_{2,0}\left(\langle .+z^k\rangle^{-1}\langle .+\zeta^{k+1}\rangle(1+\|(z^k)'_x\|)+\|(\zeta^{k+1})'_x\|\right)(x,\xi)
\\
&\leq & \|J_{k+1}\|_{2,0}\left(2\cdot h^{-1}(1+\|(z^k)'_x\|)+\|(\zeta^{k+1})'_x\|\right)(x,\xi)
\\
&\leq &2\|J_{k+1}\|_{2,0}\left[h^{-1}\left(1+\|(z^k)'_x\|+h\cdot\|(\zeta^{k+1})'_x\|\right)\right](x,\xi)
\\
&\leq & 2\|J_{k+1}\|_{2,0}\left[h^{-1}\left(1+\ds\sum_{k=1}^M\left(\|y'_{k, x}\|+h\cdot\|\eta'_{k, x}\|\right)\right)\right](x,\xi)
\\
&\leq & 2\x^{-1}\<\xi\>\|J_{k+1}\|_{2,0}\left(1+\frac{3\tau_0}{1-3\tau_0}\right)
\\
&=& C_{0,1}\|J_{k+1}\|_{2,0}\x^{-1}\<\xi\>,\quad x,\xi\in\R^n,
\eeqsn
and also
\[\|y'_{k, \xi}(x,\xi)\|\leq C_{1,0}\|J_{k}\|_{2,0}\x\<\xi\>^{-1},\qquad \|\eta'_{k, \xi}(x,\xi)\|\leq C_{1,0}\|J_{k+1}\|_{2,0},\quad x,\xi\in\R^n.
\]
The step $N=1$ is complete.
\\
\emph{Step $N\leadsto N+1$.} Let us now suppose that \eqref{stimay}, \eqref{stimaeta} hold for $1\leq |\alpha+\beta|\leq N$, $N\geq 1$, $ x,\xi\in\R^n$, and prove the same estimates for $|\alpha+\beta|=N+1.$
If we substitute \eqref{marr} into \eqref{stimay}, \eqref{stimaeta} we immediately get 
 \beqs\label{stimayind} 
|\partial_\xi^\alpha\partial_x^\beta y_k(x,\xi)|\leq c'_{\alpha,\beta} \<\xi\>^{-|\alpha|}\x^{1-|\beta|},
\\
\label{stimaetaind}
|\partial_\xi^\alpha\partial_x^\beta\eta_k(x,	\xi)|\leq c'_{\alpha,\beta}\<\xi\>^{1-|\alpha|}\x^{-|\beta|},
\eeqs
for $1\leq |\alpha+\beta|\leq N$ and $k=1,\ldots,M.$ These estimates are going to be used to bound the derivatives $\partial_x^\beta\partial_\xi^\alpha$ with $|\alpha+\beta|=N$ of the functions $y'_{k, x},\ y'_{k, \xi},\  \eta'_{k, x},\ \eta'_{k, \xi}$ (i.e. the derivatives $\partial_x^\beta\partial_\xi^\alpha$ with $|\alpha+\beta|=N+1$ of the functions $y_k, \eta_k$). Let us start by computing, from \eqref{derx}, the derivative
\beqs\label{luu}
\partial_x^\beta\partial_\xi^\alpha y'_{k, x}&=&
\partial_x^\beta\partial_\xi^\alpha\left[
J''_{k, \xi x}(.+z^{k-1},..+\zeta^{k})\cdot\left(1+(z^{k-1})'_x\right)
\right]
\\
\nonumber&+&\partial_x^\beta\partial_\xi^\alpha\left[J''_{k, \xi\xi}(.+z^{k-1},..+\zeta^{k})\cdot(\zeta^{k})'_x\right].
\eeqs
To obtain an estimate of \eqref{luu}, we use Fa\'a di Bruno formula, write the derivatives of $z^k$ and $\zeta^k$ as derivatives with respect to $y_k$ and $\eta_k$ by \eqref{zj}, and finally we apply \eqref{stimayind}, \eqref{stimaetaind}, obtaining
\beqsn
&&|\partial_x^\beta\partial_\xi^\alpha\left(J''_{k, \xi x}(x+z^{k-1}(x,\xi),\xi+\zeta^{k}(x,\xi))\right)|
\\
&&\hskip+0.3cm \leq \ds\sum_{\afrac {\beta_1+\cdots+\beta_r=\beta}{\beta_i\neq 0}} \ds\sum_{\afrac{\alpha_1+\cdots+\alpha_q=\alpha}{\alpha_i\neq 0}}C_{q,r,\alpha,\beta}\|J_k\|_{2,q+r}\<\xi\>^{-q}\x^{-r}\cdot
\\
&&\hskip+0.3cm\phantom{ \leq \ds\sum_{\afrac {\beta_1+\cdots+\beta_r=\beta}{\beta_i\neq 0}} \ds\sum_{\afrac{\alpha_1+\cdots+\alpha_q=\alpha}{\alpha_i\neq 0}}C_{q,r,\alpha,\beta}\|J_k\|_{2,q+r}}
\cdot\<\xi\>^{(1-|\alpha_1|)+\cdots+(1-|\alpha_q|)}\x^{(1-|\beta_1|)+\cdots+(1-|\beta_r|)}
\\
&&\hskip+0.6cm\leq  C_{\alpha,\beta} \|J_k\|_{2,|\alpha+\beta|}\<\xi\>^{-|\alpha|}\x^{-|\beta|}
\eeqsn
and 
\beqsn
|\partial_x^\beta\partial_\xi^\alpha\left(J''_{k, \xi\xi}(x+z^{k-1}(x,\xi),\xi+\zeta^{k}(x,\xi))\right)|\leq C_{\alpha,\beta} \|J_k\|_{2,|\alpha+\beta|}\<\xi\>^{-1-|\alpha|}\x^{1-|\beta|}.
\eeqsn
Thus, coming back to \eqref{luu}, substituting these last two estimates and using \eqref{zj} we come to
\beqs\label{541i}
\nonumber|\partial_x^\beta\partial_\xi^\alpha y'_{k, x}(x,\xi)|&\leq& \|J_k\|_{2,0}\ds\sum_{j=1}^M\left(|\partial_x^\beta\partial_\xi^\alpha y'_{j, x}(x,\xi)|+2\x\<\xi\>^{-1}|\partial_x^\beta\partial_\xi^\alpha \eta'_{j, x}(x,\xi)|\right)
\\
\nonumber&+& C'_{\alpha,\beta} \|J_k\|_{2,|\alpha+\beta|}\left(\<\xi\>^{-|\alpha|}\x^{-|\beta|}+ C''_{\alpha,\beta} \<\xi\>^{-1-|\alpha|}\x^{1-|\beta|}\<\xi\>\x^{-1}\right)
\\
\nonumber&\leq &  \|J_k\|_{2,0}\ds\sum_{j=1}^M\left(|\partial_x^\beta\partial_\xi^\alpha y'_{j, x}(x,\xi)|+2\x\<\xi\>^{-1}|\partial_x^\beta\partial_\xi^\alpha \eta'_{j, x}(x,\xi)|\right)
\\
&+&\widetilde C_{\alpha,\beta} \|J_k\|_{2,|\alpha+\beta|}\<\xi\>^{-|\alpha|}\x^{-|\beta|}.
\eeqs
Working similarly on the terms $\partial_x^\beta\partial_\xi^\alpha\eta'_{k, x}$ coming from the derivatives in \eqref{derx}, we get the corresponding estimate:
\beqs\label{541ii}
\nonumber|\partial_x^\beta\partial_\xi^\alpha \eta'_{k, x}(x,\xi)|&\leq&  \|J_{k+1}\|_{2,0}\x^{-1}\<\xi\>\ds\sum_{j=1}^M\left(|\partial_x^\beta\partial_\xi^\alpha y'_{j, x}(x,\xi)|+2\x\<\xi\>^{-1}|\partial_x^\beta\partial_\xi^\alpha \eta'_{j, x}(x,\xi)|\right)
\\
&+&\widetilde C'_{\alpha,\beta} \|J_{k+1}\|_{2,|\alpha+\beta|}\<\xi\>^{1-|\alpha|}\x^{-1-|\beta|}.
\eeqs
Now summing up for $k=1,\ldots,M$ inequalities \eqref{541i} and \eqref{541ii} we have
\beqsn
&&\ds\sum_{k=1}^M\left(|\partial_x^\beta\partial_\xi^\alpha y'_{k, x}(x,\xi)|+2\x\<\xi\>^{-1}|\partial_x^\beta\partial_\xi^\alpha \eta'_{k, x}(x,\xi)|\right)\leq 
\\
&&\hskip+0.3cm \leq\left(\ds\sum_{k=1}^M\|J_k\|_{2,0}+2\ds\sum_{k=1}^M\|J_{k+1}\|_{2,0}\right)\cdot \ds\sum_{k=1}^M\left(|\partial_x^\beta\partial_\xi^\alpha y'_{k, x}(x,\xi)|+2\x\<\xi\>^{-1}|\partial_x^\beta\partial_\xi^\alpha \eta'_{k, x}(x,\xi)|\right)
\\
&&\hskip+0.6cm+\bar C_{\alpha,\beta} \left(\ds\sum_{k=1}^M\|J_k\|_{2,|\alpha+\beta|}+2\ds\sum_{k=1}^M\|J_{k+1}\|_{2,|\alpha+\beta|}\right)\<\xi\>^{-|\alpha|}\x^{-|\beta|}
\\
&&\leq 3c_0\tau_0\ds\sum_{k=1}^M\left(|\partial_x^\beta\partial_\xi^\alpha y'_{k, x}(x,\xi)|+2\x\<\xi\>^{-1}|\partial_x^\beta\partial_\xi^\alpha \eta'_{k, x}(x,\xi)|\right)+3c_{|\alpha+\beta|}\tau_0 \bar C_{\alpha,\beta}\<\xi\>^{-|\alpha|}\x^{-|\beta|},
\eeqsn
where $c_0,c_{|\alpha+\beta|}$ are the constants defined in \eqref{marr}. In particular,
notice that, by \eqref{hyp}, we have $c_0=1$. From this, we finally obtain
\beqs\label{olla}
\nonumber\ds\sum_{k=1}^M\left(|\partial_x^\beta\partial_\xi^\alpha y'_{k, x}(x,\xi)|+2\x\<\xi\>^{-1}|\partial_x^\beta\partial_\xi^\alpha \eta'_{k, x}(x,\xi)|\right)&\leq& \bar C'_{\alpha,\beta}\frac{\tau_0}{1-3\tau_0} \<\xi\>^{-|\alpha|}\x^{-|\beta|}
\\
&<& \bar C'_{\alpha,\beta}\<\xi\>^{-|\alpha|}\x^{-|\beta|}
\eeqs
by the choice of $\tau_0$ in \eqref{1/4}. Substituting \eqref{olla} in \eqref{541i} and \eqref{541ii} we get
\beqs\label{fatto1}
|\partial_x^\beta\partial_\xi^\alpha y'_{k, x}(x,\xi)|&\leq& C_{\alpha,\beta}\|J_k\|_{2,|\alpha+\beta|}\<\xi\>^{-|\alpha|}\x^{-|\beta|}
\\
\label{fatto2}|\partial_x^\beta\partial_\xi^\alpha \eta'_{k, x}(x,\xi)|&\leq& C_{\alpha,\beta}\|J_{k+1}\|_{2,|\alpha+\beta|}\<\xi\>^{1-|\alpha|}\x^{-1-|\beta|}.
\eeqs
All the computations from \eqref{luu} to \eqref{fatto2} on the functions $y'_{k, x}$ and $ \eta'_{k, x}$ can be repeated on the functions $y'_{k, \xi}$ and $\eta'_{k, \xi}$ with minor changes. In this way we finally obtain the estimates
corresponding to \eqref{fatto1} and \eqref{fatto2}, namely 
\beqs\label{fatto3}
|\partial_x^\beta\partial_\xi^\alpha y'_{k, \xi}(x,\xi)|&\leq& C_{\alpha,\beta}\|J_k\|_{2,|\alpha+\beta|}\<\xi\>^{-1-|\alpha|}\x^{1-|\beta|}
\\
\label{fatto4}|\partial_x^\beta\partial_\xi^\alpha \eta'_{k, \xi}(x,\xi)|&\leq& C_{\alpha,\beta}\|J_{k+1}\|_{2,|\alpha+\beta|}\<\xi\>^{-|\alpha|}\x^{-|\beta|}.
\eeqs
The proof is complete, since \eqref{fatto1}-\eqref{fatto4} are the desired estimates \eqref{stimay} and \eqref{stimaeta} for all the derivatives of order $N+1$ of the functions $y_k$ and $\eta_k.$
\end{proof}    

We conclude with a Theorem that summarizes what we have proved throughout the present section, 
and gives the main properties of the multi-products of regular SG phase functions.

\begin{theorem}\label{properties}
Under assumptions \eqref{1/4} and \eqref{hyp}, the multi-product $\phi(x,\xi)$ of Definition \ref{mprod} is well defined for every $M\geq 1$ and has the following properties.
\begin{enumerate}
\item There exists $k\geq 1$ such that $\phi(x,\xi)=(\varphi_1\ \sharp\ \cdots\ \sharp\ \varphi_{M+1})(x,\xi)\in\mathcal P_r(k\bar\tau_{M+1})$ and, setting $$J_{M+1}(x,\xi):=(\varphi_1\ \sharp\ \cdots\ \sharp\ \varphi_{M+1})(x,\xi)-x\cdot\xi,$$ the sequence $\{J_{M+1}/\bar\tau_{M+1}\}_{M\geq1}$ is bounded in $S^{1,1}(\R^{2n})$.
\item The following relations hold:
\[\begin{cases}
\phi'_x(x,\xi)=\varphi'_{1, x}(x,N_1(x,\xi))
\\
\phi'_\xi(x,\xi)=\varphi'_{M+1, \xi}(Y_M(x,\xi),\xi),
\end{cases}
\]
where $(Y,N)$ is the critical point \eqref{(C)}.
\item The associative law holds:  $\varphi_1\ \sharp\ (\varphi_2\ \sharp\ \cdots\ \sharp\ \varphi_{M+1})=(\varphi_1\ \sharp\ \cdots\ \sharp\ \varphi_{M})\ \sharp\ \varphi_{M+1}$.
\item For any $\ell\geq 0$ there exist $0<\tau^\ast<1/4$ and $c^\ast\geq 1$ such that, if $\varphi_j\in \mathcal P_r(\tau_j,\ell)$ for all $j$ and $\tau_0\leq \tau^\ast$, then $\phi\in \mathcal P_r(c^\ast\bar\tau_{M+1},\ell)$.
\end{enumerate}
\end{theorem}
\begin{proof}
By theorems \ref{exists} and \ref{cinfty} we know that, for any $M\geq 1$, $\phi$ is a well-defined smooth function on 
$\R^{2n}$. We start by showing (1). We write, with $Y_{0}(x,\xi)=Y_{M+1}(x,\xi):=x$, $N_{M+1}(x,\xi):=\xi$, 
\beqsn
J_{M+1}(x,\xi)&=&\ds\sum_{j=1}^M\left(\vp_j(Y_{j-1}(x,\xi),N_j(x,\xi))-Y_j(x,\xi)\cdot N_j(x,\xi)\right)
\\
&&\hspace*{4mm}+
\vp_{M+1}(Y_M(x,\xi),\xi)-x\cdot\xi
\\
&=&\ds\sum_{j=1}^{M+1}\left(\vp_j(Y_{j-1},N_j)-Y_j\cdot N_j\right)\!(x,\xi)
\\
&=&\ds\sum_{j=1}^{M+1}\left(J_j(Y_{j-1},N_j)-(Y_j-Y_{j-1})\cdot N_j\right)\!(x,\xi).
\eeqsn
This gives that
\[\frac{J_{M+1}}{\bar\tau_{M+1}}=\ds\sum_{j=1}^{M+1}\frac{\tau_j}{\bar\tau_{M+1}}\left(\frac{J_j(Y_{j-1},N_j)}{\tau_j}-\frac{Y_j-Y_{j-1}}{\tau_j}\cdot N_j\right){\rm \ is\ bounded\ in\ } S^{1,1}
\]
since $\{J_{j}/\tau_j\}_{j\geq1}$ is bounded in $S^{1,1}$, \eqref{324} holds, and $\<N_j(x,\xi)\>\asymp\<\xi\>$.
Now, the boundedness proved here above implies the existence of a positive constant $k$ such that
\beqs\label{ue}
\| J_{M+1}\|_2\leq k\bar\tau_{M+1}<k\tau_0,
\eeqs
and taking $\tau_0$ small enough, so that $k\tau_0<1$, we obtain that $\phi\in\mathcal P_r(k\bar\tau_{M+1})$. Statement (1) is proved.
Statement (4) immediately follows. Indeed, if $\varphi_j\in \mathcal P_r(\tau_j,\ell)$, then we have
$\|J_{M+1}\|_\ell\leq (k+1)\bar\tau_{M+1}$, with $k$ coming from \eqref{ue}, and we obtain $\|J_{M+1}\|_\ell\leq c^\ast \bar\tau_{M+1}$ and $\phi\in \mathcal P_r(c^\ast\bar\tau_{M+1},\ell)$ if we choose $c^\ast$ such that $c^\ast\tau_0<1.$
\noindent
 Let us now come to (2), which is quite simple. Indeed, from \eqref{mprod} and \eqref{(C)}, we have
\[
\phi(x,\xi):=\ds\sum_{j=1}^M\left(\vp_j(Y_{j-1}(x,\xi),N_j(x,\xi))-Y_j(x,\xi)\cdot N_j(x,\xi)\right)+\vp_{M+1}(Y_M(x,\xi),\xi).
\]
A derivation of the expression above with respect to $x$ and the use of \eqref{(C)} give
\beqsn
\phi'_x(x,\xi)&\!=\!& \ds\sum_{j=1}^M\left(
\vp'_{j, x}(Y_{j-1}(x,\xi),N_{j}(x,\xi))\cdot Y'_{j-1, x}(x,\xi)\right.
\\
&\!\!& \hspace*{4.5mm}+\vp'_{j, \xi}(Y_{j-1}(x,\xi),N_{j}(x,\xi))\cdot N'_{j, x}(x,\xi)
\\
&\!\!& \hspace*{4.5mm}\left.-Y'_{j, x}(x,\xi)\cdot N_j(x,\xi)-Y_j(x,\xi)\cdot N'_{j, x}(x,\xi)
\right)
\\
&+&\vp'_{M+1, x}(Y_M(x,\xi),\xi)\!\cdot\! Y'_{M, x}(x,\xi)
\\
&\!=\!&\vp'_{1,x}(x,N_1(x,\xi))-Y'_{1, x}(x,\xi)\cdot N_1(x,\xi)
\\
&+&\ds\sum_{j=2}^M \left(( Y'_{j-1, x}\cdot  N_{j-1} -Y'_{j, x} \cdot N_j)+N_M\cdot Y'_{M, x}\right)\!(x,\xi)
\\
&=&\vp'_{1,x}(x,N_1(x,\xi)), 
\eeqsn
which is exactly the first equality in (2). The second equality can be obtained similarly, by derivation with respect to $\xi$ of 
$\phi(x,\xi)$.
\\
Finally, we deal with (3). We want to show that 
\beqs\label{wwt}
(\vp_1\ \sharp\vp_2\ \sharp\cdots\ \sharp\ \vp_M)\ \sharp\ \vp_{M+1}=\vp_1\ \sharp\cdots\ \sharp\ \vp_{M+1}.
\eeqs
To this aim, let us denote
$$\widetilde\phi:=\varphi_1\ \sharp\ \cdots\ \sharp\ \varphi_{M},$$
and compute by \eqref{serveilnome}, with $M=1$, the product
\beqs\label{issa}
(\widetilde\phi\ \sharp\ \varphi_{M+1})(x,\xi)=\widetilde\phi(x,\widetilde N(x,\xi))-
\widetilde Y(x,\xi)\cdot\widetilde N(x,\xi)+\vp_{M+1}(\widetilde Y(x,\xi),\xi),
\eeqs
where $(\widetilde Y, \widetilde N)=(\widetilde Y, \widetilde N)(x,\xi)$ is the $2n-$dimensional critical point given by
\beqs\label{341}
\begin{cases}
\widetilde Y=\widetilde\phi'_\xi(x,\widetilde N),
\\
\widetilde N=\vp'_{M+1, x}(\widetilde Y,\xi).
\end{cases}
\eeqs
Notice that $\widetilde\phi\ \sharp\ \varphi_{M+1}$ is well-defined by (1) (eventually, with a smaller $\tau_0$). 
Now, we compute the value of $\widetilde\phi(x,\widetilde N(x,\xi))=(\varphi_1\ \sharp\ \cdots\ \sharp\ \varphi_{M})(x,\widetilde N(x,\xi))$ in \eqref{issa}, using \eqref{serveilnome} with $M-1$ in place of $M$ and $\widetilde N$ in place of $\xi$, obtaining
\beqs\label{atcop}
\begin{aligned}
\widetilde\phi&(x,\widetilde N(x,\xi))=
\\&\hskip+0.8cm=\ds\sum_{j=1}^{M-1}\left(\vp_j(\bar Y_{j-1}(x,\widetilde N(x,\xi)),\bar N_j(x,\widetilde N(x,\xi)))-
\bar Y_j(x,\widetilde N(x,\xi))\cdot\bar N_j(x,\widetilde N(x,\xi))\right)
\\
&\hskip+0.8cm+\vp_{M}(\bar Y_{M-1}(x,\widetilde N(x,\xi)),\widetilde N(x,\xi)),
\end{aligned}
\eeqs
with the $2(M-1)n-$dimensional critical point $(\bar Y, \bar N)$ given by
\beqs\label{dai!}
\begin{cases}
\bar Y_0=x
\\
\bar Y_j=\vp'_{j, \xi}(\bar Y_{j-1},\bar N_{j}) & j=1,\ldots,M-1
\\
\bar N_j=\vp'_{j+1, x}(\bar Y_j,\bar N_{j+1}) & j=1,\ldots,M-1
\\
\bar N_{M}=N,
\end{cases}
\eeqs
obtained from \eqref{(C)}, with $M-1$ in place of $M$ and $\widetilde N$ in place of $\xi.$ Moreover, we have from 
\eqref{341} and (2), with $M-1$ in place of $M$, that
\beqs\label{tu}
\begin{aligned}
\widetilde Y(x,\xi)&=\widetilde\phi'_\xi(x,\widetilde N(x,\xi))=
(\varphi_1\ \sharp\ \cdots\ \sharp\ \varphi_{M})'_\xi(x,\widetilde N(x,\xi))
\\
&=
\vp'_{M, \xi}(\bar Y_{M-1}(x,\widetilde N(x,\xi)), \widetilde N(x,\xi)).
\end{aligned}
\eeqs
Summing up, from \eqref{tu}, the second equation in \eqref{341}, and \eqref{dai!}, we have that
$(\bar Y_1,\cdots,\bar Y_{M-1}, \widetilde Y, \bar N_1, \cdots, \bar N_{M-1}, \widetilde N)$ solves system \eqref{(C)}, and thus it is the $2Mn-$di\-men\-sio\-nal critical point needed to define the multi-product $\varphi_1\ \sharp\ \cdots\ \sharp\ \varphi_{M+1}$, which turns out to be given, in view of \eqref{serveilnome}, by
\begin{align*}
(\varphi_1\ \sharp\ &\cdots\ \sharp\ \varphi_{M+1})(x,\xi)=
\\
&=\ds\sum_{j=1}^{M-1}\left(\vp_j(\bar Y_{j-1}(x,\widetilde N(x,\xi)),\bar N_j(x,\widetilde N(x,\xi)))
-\bar Y_j(x,\widetilde N(x,\xi))\cdot\bar N_j(x,\widetilde N(x,\xi))\right)
\\
&+\vp_{M}(\bar Y_{M-1}(x,\widetilde N(x,\xi)),\widetilde N(x,\xi))-\bar Y(x,\xi)\cdot\bar N(x,\xi)+
\vp_{M+1}(\widetilde Y(x,\xi),\xi).
\end{align*}
We observe that this last expression coincides with \eqref{issa} after substituting \eqref{atcop} in it. This gives that $\varphi_1\ \sharp\ \cdots\ \sharp\ \varphi_{M+1}=\widetilde\phi\ \sharp\ \vp_{M+1}$, that is \eqref{wwt}. Similarly, we can prove the corresponding law $\varphi_1\ \sharp\ \left(\varphi_2\ \sharp\cdots\ \sharp\ \varphi_{M+1}\right)=\varphi_1\ \sharp\cdots\ \sharp\ \varphi_{M+1}$, completing the proof of (3).
\end{proof}

\section{Composition of SG Fourier integral operators}\label{sec:sgfioprod}
\setcounter{equation}{0}
%
%
We can now prove our main theorem on compositions of regular SG FIOs. 
We start with an invertibility result for $I_\varphi=\Op_\varphi(1)$ and $I^*_\varphi=\Op^*_\varphi(1)$
when $\varphi$ is a regular phase function. Theorem \ref{thm:parama} below gives more precise versions
of \eqref{eq:parami}, \eqref{eq:paramii}, with a slight additional restriction on $\fy$, for FIOs with constant, nonvanishing
symbol. 
\begin{theorem} \label{thm:parama}
	Assume that $\varphi\in\Phr(\tau)$ with $0<\tau<\frac{1}{4}$ sufficiently small. Then, there exists $q\in S^{0,0}(\R^{2n})$ such that
	\begin{align}
		\label{eq:parama1}
		I_\varphi\circ\Op^*_\varphi(q)=\Op^*_\varphi(q)\circ I_\varphi=I,
		\\
		\label{eq:parama2}
		I^*_\varphi\circ\Op_\varphi(q)=\Op_\varphi(q)\circ I^*_\varphi=I.
	\end{align}
	Moreover, if the family of SG phase functions $\{\fy_s(x,\xi)\}$ is such that the  family $\{J_s(x,\xi)\}=\{\fy_s(x,\xi)-x\cdot\xi\}$ 
	is bounded in $S^{1,1}$, then the corresponding family $\{q_s\}$ is also bounded in $S^{0,0}$.
\end{theorem}
\begin{proof}
	For $u\in\SX(\R^n)$ we have, by definition of type I and type II SG FIOs, 
	\begin{equation}\label{eq:compi-ii}
		((I_\varphi\circ I^*_\varphi) u)(x)=(2\pi)^{-n}\iint e^{i(\varphi(x,\xi)-\varphi(y,\xi))}\,u(y)\,dy d\xi.
	\end{equation}
	The map
	\[
		\Xi_{x,y}\colon\xi\mapsto\Xi_{x,y}(\xi)=\Xi(x,y,\xi)=\int_0^1 \fy^\prime_x(x+t(y-x),\xi)\,dt 		
	\]
	is globally invertible on $\R^n$. In fact, its Jacobian is given by the matrix
	\[
		\int_0^1 \fy^{\prime\prime}_{x\xi}(x+t(y-x),\xi)\,dt=I+\int_0^1 J^{\prime\prime}_{x\xi}(x+t(y-x),\xi)\,dt  
	\]
	which has nonvanishing determinant, in view of the hypothesis $\fy\in\Phr(\tau)$, $0\le\tau<\frac{1}{4}$. 
	Moreover, condition (2) in Definition \ref{def:phase} implies that $\Xi$ is coercive, and these two
	properties give its global invertibility on $\R^n$, see \cite[Theorems 11 and 12]{Coriasco:998.1} 
	and the references
	quoted therein. Finally, $\Xi_{x,y}$ is also a \textit{SG diffeomorphism with $0$-order parameter-dependence}, that is
	both $\Xi(x,y,\xi)$ and $\Xi^{-1}(x,y,\eta)$ belong to $S^{0,0,1}(\R^{3n})$,
	the space of \textit{SG amplitudes} of order $(0,0,1)$, see \cite{CO,Coriasco:998.1}, 
	and satisfy $\norm{\Xi(x,y,\xi)}\asymp\norm{\xi}$,
	$\norm{\Xi^{-1}(x,y,\eta)}\asymp\norm{\eta}$, uniformly with respect to $x,y\in\R^n$.
	In \eqref{eq:compi-ii} we can then change variable, setting
	\[
		\eta=\Xi(x,y,\xi)\Leftrightarrow \xi=\Xi^{-1}(x,y,\eta),
	\]
	and obtain
	\[
		((I_\varphi\circ I^*_\varphi) u)(x)=u(x)+(2\pi)^{-n}\iint e^{i(x-y)\cdot\eta}a_0(x,y,\eta)u(y)\,dyd\eta=((I+A_0) u)(x),
	\]
	with
	\[
		a_0(x,y,\eta)=\det(I+J^{\prime\prime}_{x\xi}(x,y,\xi))^{-1}|_{\xi=\Xi^{-1}(x,y,\eta)}-1.
	\]
	By the results on composition of \textit{SG functions} in \cite{Coriasco:998.1, CT14},
	we find that $a_0\in S^{0,0,0}(\R^{3d})$, the space of SG-amplitudes of order $(0,0,0)$.
	Since the seminorms of $a_0$ can be controlled by means of the parameter $\tau$,
	and the map associating $a_0$ with the symbol $a\in S^{0,0}$ such that $A_0=\Op(a)$
	is continuous, the same holds for the seminorms of $a$. By general arguments, see 
	\cite{CO,Kumano-go:1,SC98,SH87},
	it turns out that $(I+\Op(a))^{-1}$ exists in $\Op(S^{0,0})$. Then, setting $Q_\fy^*=I^*_\fy\circ (I+\Op(a))^{-1}$,
	using Theorem \ref{thm:compi} we find $Q_\fy^*=\Op^*_\fy(q)$ for some $q\in S^{0,0}$ and
	$I_\fy\circ\Op^*_\fy(q)=I$, which is the first part of \eqref{eq:parama1}. The remaining statements follow
	by arguments analogous to those used in the proof of \cite[Theorem 6.1]{Kumano-go:1}.
\end{proof}

The next Theorem \ref{thm:mainbis} is one of our main results.

\begin{theorem}\label{thm:mainbis}
	Let $\fy_j\in\Phr(\tau_j)$, $j=1,2$, be such that $0\le\tau_1+\tau_2\le\tau\le\frac{1}{4}$ for
	some sufficiently small $\tau>0$. Then, there
	exists $p\in S^{0,0}(\R^{2n})$ such that 
	\begin{align}
		\label{eq:maini}
		I_{\fy_1}\circ I_{\fy_2} &= \Op_{\fy_1\sharp\fy_2}(p),
		\\
		\label{eq:mainii}
		I^*_{\fy_2}\circ I^*_{\fy_1}&=\Op^*_{\fy_1\sharp\fy_2}(p).
	\end{align}
	Moreover, if the families of SG phase functions $\{\fy_{js}(x,\xi)\}$, $j=1,2$, are such that the families
	$\{J_{js}(x,\xi)\}=\{\fy_{js}(x,\xi)-x\cdot\xi\}$ are bounded in $S^{1,1}$, $j=1,2$, then, the 
	corresponding family $\{p_{s}(x,\xi)\}$ is also bounded in $S^{0,0}$.
\end{theorem}
We will achieve the proof of Theorem \ref{thm:mainbis} through various intermediate results,
adapting the analogous scheme in \cite{Kumano-go:1}. Before getting to that, let us first state and 
prove our main Theorem \ref{thm:main}, which is obtained as a consequence of Theorems
\ref{thm:parama} and \ref{thm:mainbis}.
\begin{theorem}\label{thm:main}
	Let $\fy_j\in\Phr(\tau_j)$, $j=1,2, \dots, M$, $M\ge 2$, 
	be such that $\tau_1+\cdots+\tau_M\le\tau\le\frac{1}{4}$
	for some sufficiently small $\tau>0$, and set
	\begin{align*}
		\Phi_0(x,\xi)&=x\cdot\xi, 
		\\
		\Phi_1&=\fy_1, 
		\\
		\Phi_j&=\fy_1\sharp\cdots\sharp\fy_j, \; j= 2, \dots, M
		\\ 
		\Phi_{M,j}&=\fy_j\sharp\fy_{j+1}\sharp\cdots\sharp\fy_{M}, j=1,\dots,M-1,
		\\
		\Phi_{M,M}&=\fy_{M}, 
		\\
		\Phi_{M,M+1}(x,\xi)&=x\cdot\xi.
	\end{align*}
	Assume also $a_j\in S^{m_j,\mu_j}(\R^{2n})$, and set $A_j=\Op_{\fy_j}(a_j)$, $j=1,\dots,M$. 
	Then, the following holds true.
	\begin{enumerate}
		\item Given $q_j, q_{M,j}\in S^{0,0}(\R^{2n})$, $j=1,\dots, M$, such that
		\[
			\Op^*_{\Phi_j}(q_j)\circ I_{\Phi_j}=I, \quad I^*_{\Phi_{M,j}}\circ\Op_{\Phi_{M,j}}(q_{M,j})=I,
		\] 
		set $Q_j^*=\Op^*_{\Phi_j}(q_j)$, $Q_{M,j}=\Op_{\Phi_{M,j}}(q_{M,j})$, and
		\[
			R_j=I_{\Phi_{j-1}}\circ A_j\circ Q^*_j, \quad
			R_{M,j}=Q_{M,j}\circ A_j\circ I^*_{\Phi_{M,j+1}}, \quad j=1,\dots,M.
		\]
		Then, $R_j,R_{M,j}\in\Op(S^{0,0}(\R^{2n}))$, $j=1,\dots,M$, and
		\begin{equation}\label{eq:prodAj}
			A=A_1\circ\cdots\circ A_M=R_1\circ\cdots\circ R_{M}\circ I_{\Phi_M}
			=I^*_{\Phi_{M,1}}\circ R_{M,1}\circ\cdots\circ R_{M,M}.
		\end{equation}
		\item There exists $a\in S^{m,\mu}(\R^{2n})$, $m=m_1+\cdots+m_M$,
		$\mu=\mu_1+\cdots+\mu_M$ such that, setting $\phi=\fy_1\sharp\cdots\sharp\fy_M$,
		\[
			A=A_1\circ\cdots\circ A_M=\Op_{\phi}(a).
		\]
		\item For any $l\in\Z_+$ there exist $l^\prime\in\Z_+$, $C_l>0$ such that 
		\begin{equation}\label{eq:estsna}
			\vvvert a \vvvert_l^{m,\mu} \le C_l\prod_{j=1}^M \vvvert a_j \vvvert_{l^\prime}^{m_j,\mu_j}.
		\end{equation}
	\end{enumerate}
\end{theorem}
\begin{proof}
	The existence of $q_j, q_{M,j}\in S^{0,0}$, $j=1,\dots, M$, with the desired properties
	follows from Theorem \ref{thm:parama}. We also notice that, trivially, $I_{\Phi_0}=I_{\Phi_{M,M+1}}=I$, so that,
	inserting either $I=Q_1^*\circ I_{\Phi_1}=\cdots=Q^*_M\circ I_{\Phi_M}$ or
	$I=I_{\Phi_{M,1}}^*\circ Q_{M,1}=\cdots=I_{\Phi_{M,M}}^*\circ Q_{M,M}$, we indeed find
	\begin{align*}
		A_1\circ\cdots\circ A_M&= I_{\Phi_0}\circ A_1\circ Q_1^*\circ I_{\Phi_1}\circ A_2\circ\cdots
		\circ I_{\Phi_M}\circ A_M\circ Q^*_{M}\circ I_{\Phi_M}
		\\
		&=R_1\circ\cdots\circ R_M\circ I_{\Phi_M}
		\\
		&=I^*_{\Phi_{M,1}}\circ Q_{M,1}\circ A_1\circ I^*_{\Phi_{M,2}}\circ Q_{M,2}\circ A_2\circ\cdots
		\circ I^*_{\Phi_{M,M}}\circ Q_{M,M}\circ A_M\circ I_{\Phi_{M,M+1}}
		\\
		&= I^*_{\Phi_{M,1}}\circ R_{M,1}\circ\cdots\circ R_{M,M},
	\end{align*}
	as claimed. Now, we observe that, again in view of Theorem \ref{thm:parama}, there exists $p_j\in S^{0,0}$
	such that $I_{\fy_j}\circ \Op^*_{\fy_j}(p_j)=I$, $j=1,\dots,M$. Setting $P_j^*=\Op_{\fy_j}^*(p_j)$, 
	and inserting it into the definition of $R_j$, by Theorem \ref{thm:mainbis} we then find, for $j=1,\dots,M$,
	\[
		R_j=(I_{\Phi_{j-1}}\circ I_{\fy_j})\circ(P_j^*\circ A_j)\circ Q_j^*
		=I_{\Phi_{j-1}\sharp\fy_j}\circ(P_j^*\circ A_j)\circ Q_j^*=I_{\Phi_j}\circ(P_j^*\circ A_j)\circ Q_j^*.
	\]
	Theorem \ref{thm:compii} implies that $P_j^*\circ A_j\in\Op(S^{m_j,\mu_j})$, and Theorem \ref{thm:compi} 
	then implies that $(P_j^*\circ A_j)\circ Q_j^*=\Op^*_{\Phi_j}(d_j)$, for some $d_j\in S^{m_j,\mu_j}$,
	$j=1,\dots, M$. Another application of Theorem \ref{thm:compi} gives that 
	\[
		R_j=I_{\Phi_j}\circ\Op^*_{\Phi_j}(d_j)\in\Op(S^{m_j,\mu_j}), j=1,\dots, M,
	\]
	so that the standard composition rules for SG pseudodifferential operators and a further application of
	Theorem \ref{thm:compi} imply, for $\phi=\fy_1\sharp\cdots\sharp\fy_M$ and a suitable $a\in S^{m,\mu}$,
	\[
		A=A_1\circ\cdots\circ A_M=\Op_\phi(a),
	\]
	as claimed. Similar considerations hold for $R_{M,j}$, $j=1, \dots, M$ and the representation formula
	\[
		A_1\circ\cdots\circ A_M=I^*_{\Phi_{M,1}}\circ R_{M,1}\circ\cdots\circ R_{M,M}.
	\]
	The estimate \eqref{eq:estsna} follows from the composition results in \cite{Coriasco:998.1}, applied repeatedly
	to \eqref{eq:prodAj}, observing that the amplitudes of the resulting operators depend continuously on those of the
	involved factors. The proof is complete.
\end{proof}
To start proving Theorem \ref{thm:mainbis}, with two SG phase functions $\varphi_1,\varphi_2$ as in the
corresponding hypotheses and $u\in\SX(\R^n)$, let us write, as it is possible,
\[
	[(I_{\varphi_1}\circ I_{\varphi_2})u](x)=\iiint e^{i(\varphi_1(x,\xi^\prime)-x^\prime\cdot\xi^\prime+\varphi_2(x^\prime,\xi))}
	\,\widehat{u}(\xi)\,\dcut\xi^\prime dx^\prime \dcut\xi.
\]
Now, with $\phi=\varphi_1\sharp\varphi_2$, set
\beqs\label{fizero}
	\varphi_0(x,x^\prime,\xi^\prime,\xi)=\varphi_1(x,\xi^\prime)-x^\prime\cdot\xi^\prime+\varphi_2(x^\prime,\xi)-\phi(x,\xi),
\eeqs
and consider, in the sense of oscillatory integrals,
\begin{equation}\label{eq:defsymbp}
	p(x,\xi)=\iint e^{i\varphi_0(x,x^\prime,\xi^\prime,\xi)}\,\dcut\xi^\prime dx^\prime.
\end{equation}
Then, we can write
\[
	[(I_{\varphi_1}\circ I_{\varphi_2})u](x)=\int e^{i\phi(x,\xi)}\,p(x,\xi)\,\widehat{u}(\xi)\,\dcut\xi,\quad u\in\SX(\R^n),
\]
which gives the desired claim, if we show that \eqref{eq:defsymbp} indeed defines a symbol $p\in S^{0,0}(\R^{2n})$. 
Let us now define the adapted cut-off functions which will be needed for the proof of this fact.
\begin{definition}\label{def:dblcutoff}
We set
\[
	\chi(x,x^\prime,\xi^\prime,\xi)=\chi_a(x,x^\prime)\cdot \chi_a(\xi,\xi^\prime),
\]
where, with $a>0$ to be fixed later and $w,w^\prime\in\R^n$, we assume
\[
	\chi_a(w,w^\prime)=\psi(a(w-w^\prime)\norm{w}^{-1}),
\]
for a fixed cut-off function $\psi\in C_0^\infty(\R^n)$. In particular, we also assume that, for all $w\in\R^n$, $0\le\psi(w)\le1$,
$\supp\,\psi = B_\frac{2}{3}(0)$, $\psi|_{B_\frac{1}{2}(0)}\equiv1$, $w\notin B_\frac{1}{2}(0)\Rightarrow 0\le\psi(w)<1$,
where $B_r(w_0)$ is the closed ball in $\R^n$ centred at $w_0$ with radius $r>0$. 
\end{definition}
For the proof of the next lemma see, e.g., \cite{Coriasco:998.1}.
\begin{lemma}\label{lem:chisym}
	 i) For any multiindeces $\gamma_1,\gamma_2\in\Z_+^n$, the function $\chi_a(w,w^\prime)$ introduced in Definition
	\ref{def:dblcutoff} satisfies, for all $w,w^\prime\in\R^n$,
	\begin{equation}\label{eq:chisym1}
		|\partial^{\gamma_1+\gamma_2}_{w^\prime}\chi_a(w,w^\prime)|\lesssim
		\norm{w}^{-|\gamma_1|}\norm{w^\prime}^{-|\gamma_2|}.
	\end{equation}
	ii) For any multiindeces $\alpha_1\alpha_2,\beta_1,\beta_2\in\Z_+^n$, 
	the function $\chi(x,x^\prime,\xi^\prime,\xi)$ introduced in 
	Definition \ref{def:dblcutoff} satisfies, for all $x,x^\prime,\xi,\xi^\prime$, the estimates
	\begin{equation}\label{eq:chisym2}
		|\partial^{\alpha_1+\alpha_2}_{x^\prime}\partial^{\beta_1+\beta_2}_{\xi^\prime}\chi(x,x^\prime,\xi^\prime,\xi)|
		\lesssim 
		\norm{x}^{-|\alpha_1|}\norm{x^\prime}^{-|\alpha_2|}
		\norm{\xi}^{-|\beta_1|}\norm{\xi^\prime}^{-|\beta_2|}.
	\end{equation}
\end{lemma}
\begin{remark}\label{rem:dblcutoff}
In view of Definition \ref{def:dblcutoff}, 
\begin{align*}
	1-\chi(x,x^\prime,\xi^\prime,\xi)&=1-\chi_a(x,x^\prime)+\chi_a(x,x^\prime)-\chi_a(x,x^\prime)\cdot \chi_a(\xi,\xi^\prime)
	\\
	&=1-\chi_a(x,x^\prime)+\chi_a(x,x^\prime)\cdot(1-\chi_a(\xi,\xi^\prime)),
\end{align*}
which implies that on $\supp(1-\chi(x,x^\prime,\xi^\prime,\xi))$ either $|x-x^\prime|\ge \dfrac{1}{2a}\norm{x}$ or 
$|\xi-\xi^\prime|\ge \dfrac{1}{2a}\norm{\xi}$.
\end{remark}
Now write $p$ in \eqref{eq:defsymbp} as $p=p_0+p_\infty$ with
\begin{align}
	\label{eq:defp0}
	p_0(x,\xi)&=
	\iint e^{i\varphi_0(x,x^\prime,\xi^\prime,\xi)}\,\chi(x,x^\prime,\xi^\prime,\xi)\,\dcut\xi^\prime dx^\prime,
	\\
	\label{eq:defpinf}
	p_\infty(x,\xi)&=
	\iint e^{i\varphi_0(x,x^\prime,\xi^\prime,\xi)}\,(1-\chi(x,x^\prime,\xi^\prime,\xi))\,\dcut\xi^\prime dx^\prime.
\end{align}
We analyze separately $p_0$ and $p_\infty$.
\begin{proposition}\label{prop:pinf}
	Under the hypotheses of Theorem \ref{thm:mainbis}, 
	for $p_\infty$ defined in \eqref{eq:defpinf} we have $p_\infty\in S^{-\infty,-\infty}(\R^{2n})$.
\end{proposition}
\begin{proof}
	Define
\[\varphi_\infty(x,x^\prime,\xi^\prime,\xi)=
		\varphi_1(x,\xi^\prime)-x^\prime\cdot\xi^\prime+\varphi_2(x^\prime,\xi)-x\cdot\xi,\]
		so we have from \eqref{fizero}
\[ \varphi_0(x,x^\prime,\xi^\prime,\xi)=\varphi_\infty(x,x^\prime,\xi^\prime,\xi)+x\cdot\xi-\phi(x,\xi)\]
and \[p_\infty(x,\xi)=e^{-iJ(x,\xi)}p_\infty^\prime(x,\xi),\]
	where we have set $J(x,\xi)=\phi(x,\xi)-x\cdot\xi$ and 
	\[\widetilde{p}_\infty(x,\xi)=
		\iint e^{i\varphi_\infty(x,x^\prime,\xi^\prime,\xi)}(1-\chi(x,x^\prime,\xi^\prime,\xi))\,\dcut\xi^\prime dx^\prime.
	\]
	It is straightforward, since $J\in S^{1,1}$ for small $\tau>0$,
	that it is enough to prove that $\widetilde{p}_\infty\in S^{-\infty,-\infty}$ to achieve the desired result. Also, in view of the
	definition of $\varphi_\infty$,
	\begin{align*}
		\varphi^\prime_{\infty,x}(x,x^\prime,\xi^\prime,\xi)&=\xi^\prime-\xi+J^\prime_{1,x}(x,\xi^\prime),
		\\
		\varphi^\prime_{\infty,\xi^\prime}(x,x^\prime,\xi^\prime,\xi)&=x-x^\prime+J^\prime_{1,\xi}(x,\xi^\prime),
		\\
		\varphi^\prime_{\infty,x^\prime}(x,x^\prime,\xi^\prime,\xi)&=\xi-\xi^\prime+J^\prime_{2,x}(x^\prime,\xi),
		\\
		\varphi^\prime_{\infty,\xi}(x,x^\prime,\xi^\prime,\xi)&=x^\prime-x+J^\prime_{2,\xi}(x^\prime,\xi).		
	\end{align*}
	Then, on $\supp(1-\chi(x,x^\prime,\xi^\prime,\xi))$, for a known $c>0$ and a sufficiently small $\tau>0$,
	depending on $\varphi_1$, $\varphi_2$, and $\chi$, there exist suitable $k_1,k_2>0$, such that either
	\begin{align*}
		|\varphi^\prime_{\infty,x^\prime}(x,\xi^\prime,x^\prime,\xi)|&\ge |\xi-\xi^\prime|-c\tau\norm{\xi}
		\ge  |\xi-\xi^\prime|-c\tau |\xi-\xi^\prime|=(1-c\tau) |\xi-\xi^\prime|
		\\
		&\ge k_1(\norm{\xi}+\norm{\xi^\prime})>0,
	\end{align*}
	or
	\begin{align*}
		|\varphi^\prime_{\infty,\xi^\prime}(x,\xi^\prime,x^\prime,\xi)|&\ge |x-x^\prime|-c\tau\norm{x}
		\ge |x-x^\prime|-c\tau|x-x^\prime|=(1-c\tau)|x-x^\prime|
		\\
		&\ge k_2(\norm{x}+\norm{x^\prime})>0.
	\end{align*}
	Let us set, for $b>2a>0$,
	\begin{align}
		\label{eq:p1}
		\widetilde{p}_{1\infty}(x,\xi)
		&=\iint e^{i\varphi_\infty(x,\xi^\prime,x^\prime,\xi)}\,(1-\chi(x,x^\prime,\xi^\prime,\xi))\cdot\chi_b(x,x^\prime)
		\,\dcut\xi^\prime dx^\prime,
		\\
		\label{eq:p2}
		\widetilde{p}_{2\infty}(x,\xi)&=
		\iint e^{i\varphi_\infty(x,\xi^\prime,x^\prime,\xi)}\,(1-\chi(x,x^\prime,\xi^\prime,\xi))\cdot(1-\chi_b(x,x^\prime))
		\cdot\chi_b(\xi,\xi^\prime)\,\dcut\xi^\prime dx^\prime,
		\\
		\label{eq:p3}
		\widetilde{p}_{3\infty}(x,\xi)&=
		\iint e^{i\varphi_\infty(x,\xi^\prime,x^\prime,\xi)}\,(1-\chi(x,x^\prime,\xi^\prime,\xi))\cdot(1-\chi_b(x,x^\prime))
		\cdot(1-\chi_b(\xi,\xi^\prime))\,\dcut\xi^\prime dx^\prime,
	\end{align}
	so that \[\widetilde{p}_\infty(x,\xi)=\widetilde{p}_{1\infty}(x,\xi)+\widetilde{p}_{2\infty}(x,\xi)
		+\widetilde{p}_{3\infty}(x,\xi).\]
	Then, the operator
	\[T_V=-i|\varphi^\prime_{\infty,x^\prime}(x,x^\prime,\xi^\prime,\xi)|^{-2}\,
		\varphi^\prime_{\infty,x^\prime}(x,x^\prime,\xi^\prime,\xi)\cdot\nabla_{x^\prime}
		=V(x,x^\prime,\xi^\prime,\xi)\cdot\nabla_{x^\prime}\]
	such that
	\[T_V e^{i\varphi_\infty(x,\xi^\prime,x^\prime,\xi)}=e^{i\varphi_\infty(x,\xi^\prime,x^\prime,\xi)}\]
	is well defined on the support if the integrand of \eqref{eq:p1}, and, respectively, the operator
	\[T_C=-i|\varphi^\prime_{\infty,\xi^\prime}(x,x^\prime,\xi^\prime,\xi)|^{-2}\,
		\varphi^\prime_{\infty,\xi^\prime}(x,x^\prime,\xi^\prime,\xi)\cdot\nabla_{\xi^\prime}
		=C(x,x^\prime,\xi^\prime,\xi)\cdot\nabla_{\xi^\prime}\]
such that
\[T_C e^{i\varphi_\infty(x,\xi^\prime,x^\prime,\xi)}=e^{i\varphi_\infty(x,\xi^\prime,x^\prime,\xi)}\]
	is well defined on the support of the
	integrand of \eqref{eq:p2}. Both $T_V$ and $T_C$ are well defined on the support of the
	integrand of \eqref{eq:p3}. Notice also that the coefficients of $T_V$ satisfy,
	on the support of the integrand of \eqref{eq:p1}, estimates of the type 
	\begin{equation}\label{eq:coeffest1}
		|\partial^\alpha_{x^\prime}\partial^\beta_{\xi^\prime}V(x,x^\prime,\xi^\prime,\xi)|\lesssim 
	 	\norm{x^\prime}^{-|\alpha|}\norm{\xi^\prime}^{-|\beta|}(\norm{\xi}+\norm{\xi^\prime})^{-1}.
	\end{equation}
	Since there $\norm{x}\asymp\norm{x^\prime}$, the same holds with $x$ in place of $x^\prime$. Similarly, the coefficients
	of $T_C$ satisfy, on the support of the integrand of \eqref{eq:p2}, estimates of the type 
	\begin{equation}\label{eq:coeffest2}
		|\partial^\alpha_{x^\prime}\partial^\beta_{\xi^\prime}C(x,x^\prime,\xi^\prime,\xi)|\lesssim
		\norm{x^\prime}^{-|\alpha|}\norm{\xi^\prime}^{-|\beta|}(\norm{x}+\norm{x^\prime})^{-1},
	\end{equation}
	as well as the analogous ones with $\xi$ in place of $\xi^\prime$, since $\norm{\xi}\asymp\norm{\xi^\prime}$ there. 
	Moreover, both \eqref{eq:coeffest1} and \eqref{eq:coeffest2} hold on the support of the intengrand in \eqref{eq:p3}.
	The claim then follows by repeated integration by parts, using $T_C$ and/or $T_V$ in the expressions of
	$p_{3\infty}$, $p_{2\infty}$, and $p_{1\infty}$, and recalling Lemma \ref{lem:chisym}.
\end{proof}

\begin{proposition}\label{prop:p0}
	Under the hypotheses of Theorem \ref{thm:mainbis}, 
	for $p_0$ defined in \eqref{eq:defp0} we have $p_0\in S^{0,0}(\R^{2n})$.
\end{proposition}

To prove Proposition \ref{prop:p0}, we will use the change of variables
\begin{equation}\label{eq:chvar}
	\begin{cases}
		x^\prime=Y(x,\xi)+y\cdot\omega(x,\xi)^{-1}
		\\
		\xi^\prime=N(x,\xi)+\eta\cdot\omega(x,\xi),
	\end{cases}
\end{equation}
where $\omega(x,\xi)=\norm{x}^{-\frac{1}{2}}\norm{\xi}^\frac{1}{2}\in S^{-\frac{1}{2},\frac{1}{2}}$ and
$(Y, N)=(Y(x,\xi),N(x,\xi))$ is the unique solution of
\begin{equation*}
	\begin{cases}
		Y(x,\xi)=\varphi^\prime_{1\xi}(x,N(x,\xi))
		\\
		N(x,\xi)=\varphi^\prime_{2x}(Y(x,\xi),\xi),
	\end{cases}
\end{equation*}
see \eqref{(C1)} of  Section \ref{sec:mpsgphf} above. With $\chi$ as in Definition \ref{def:dblcutoff}, let
\begin{align*}
	\rho(y,\eta;x,\xi)&
	=\chi(x, Y(x,\xi)+y\cdot\omega(x,\xi)^{-1}, N(x,\xi)+\eta\cdot\omega(x,\xi), \xi),
	\\
	\varphi(y,\eta;x,\xi)&
	=\varphi_0(x, Y(x,\xi)+y\cdot\omega(x,\xi)^{-1}, N(x,\xi)+\eta\cdot\omega(x,\xi), \xi),
\end{align*}
so that
\begin{equation*}
	p_0(x,\xi)=\iint e^{\varphi(y,\eta;x,\xi)}\rho(y,\eta;x,\xi)\,dy\dcut\eta.
\end{equation*}
By construction, on $\supp\,\rho$, 
\begin{equation*}
	|Y(x,\xi)+y\cdot\omega(x,\xi)^{-1}-x|\le\frac{2}{3a}\norm{x},	
	\quad
	|N(x,\xi)+\eta\cdot\omega(x,\xi)-\xi|\le\frac{2}{3a}\norm{\xi},
\end{equation*}
which implies that, for a sufficiently large $a>0$ and a suitable $\tilde{k}\in(0,1)$, on $\supp\,\rho$ we also have by \eqref{simplificationH} and \eqref{zzeta}
\begin{equation*}
	|y|\cdot\omega(x,\xi)^{-1}\le\tilde{k}\norm{x}\text{\; and\; } |\eta|\cdot\omega(x,\xi)\le\tilde{k}\norm{\xi}
	\;\Rightarrow\; |y|,|\eta|\le \tilde{k}\,(\norm{x}\norm{\xi})^\frac{1}{2}.
\end{equation*}
Furthermore, recalling that $\norm{\varphi^\prime_{1\xi}(x,\xi)}\asymp\norm{x}$ and 
$\norm{\varphi^\prime_{2x}(x,\xi)}\asymp\norm{\xi}$, we find that, on $\supp\,\rho$, for any $\theta\in[0,1]$,
\begin{equation}\label{eq:equiv}
		\norm{Y(x,\xi)+\theta\cdot y\cdot\omega(x,\xi)^{-1}}\asymp\norm{x},
		\quad
		\norm{N(x,\xi)+\theta\cdot\eta\cdot\omega(x,\xi)}\asymp\norm{\xi}.
\end{equation}
The next Lemma \ref{lem:ab} can be proved analysing the Taylor expansions of $\varphi(y,\eta;x,\xi)$.
\begin{lemma}\label{lem:ab}
Let
\begin{align*}
	A_1(\eta;x,\xi)&=\omega(x,\xi)^2\int_0^1(1-\theta)
	J^{\prime\prime}_{1\xi\xi}(x,N(x,\xi)+\theta\cdot\eta\cdot\omega(x,\xi))\,d\theta,
	\\
	A_2(y;x,\xi)&=\omega(x,\xi)^{-2}\int_0^1(1-\theta)
	J^{\prime\prime}_{2xx}(Y(x,\xi)+\theta\cdot y\cdot\omega(x,\xi)^{-1},\xi)\,d\theta,
	\\
	B_1(\eta;x,\xi)&=\omega(x,\xi)^2\int_0^1
	J^{\prime\prime}_{1\xi\xi}(x,N(x,\xi)+\theta\cdot\eta\cdot\omega(x,\xi))\,d\theta,
	\\
	B_2(y;x,\xi)&=\omega(x,\xi)^{-2}\int_0^1
	J^{\prime\prime}_{1xx}(Y(x,\xi)+\theta\cdot y\cdot\omega(x,\xi)^{-1},\xi)\,d\theta.
\end{align*}
Then
\begin{align}\nonumber
	\varphi(y,\eta;x,\xi)=&-y\cdot\eta+(\varphi_1(x,N(x,\xi)+\eta\cdot\omega(x,\xi))-\varphi_1(x, N(x, \xi)))
	\\
	\nonumber
	&-\varphi^\prime_{1\xi}(x,N(x,\xi))\cdot\eta\cdot\omega(x,\xi)	\\
	\nonumber
	&+(\varphi_2(Y(x,\xi)+y\cdot\omega(x,\xi)^{-1},\xi)-\varphi_2(Y(x,\xi),\xi))
	\\
	\label{biss}
	&-\varphi^\prime_{2x}(Y(x,\xi),\xi))\cdot y\cdot\omega(x,\xi)^{-1}
	\\
	\nonumber
	=&-y\cdot\eta+[A_1(\eta;x,\xi)\eta]\cdot\eta+[A_2(y;x,\xi)y]\cdot y,
\\
\nonumber
\\
\nonumber
	\varphi^\prime_y(y,\eta;x,\xi)=&-\eta
	+[\varphi^\prime_{2x}(Y(x,\xi)+y\cdot\omega(x,\xi)^{-1},\xi)-\varphi^\prime_{2x}(Y(x,\xi),\xi)]\cdot\omega(x,\xi)^{-1}
	\\
	\label{biss2}
	=&-\eta+B_2(y;x,\xi)y,
\\
\nonumber
\\
\nonumber
	\varphi^\prime_\eta(y,\eta;x,\xi)=&-y
	+[\varphi^\prime_{1\xi}(x,N(x,\xi)+\eta\cdot\omega(x,\xi))-\varphi^\prime_{1\xi}(x,N(x,\xi))]\cdot\omega(x,\xi)
	\\
	\label{biss3}
	=&-y+B_1(\eta;x,\xi)\eta.
\end{align}
\end{lemma}
\begin{proof}\label{rem:phi0}
	By the definition \eqref{fizero} of $\varphi_0$ and of the multi-product of phase functions \eqref{serveilnome} and \eqref{3.4}, recalling \eqref{eq:chvar}, we can write
	\begin{align*}
		\varphi_0(x,x^\prime,\xi^\prime,\xi)&=\varphi_1(x,\xi^\prime)-x^\prime\cdot\xi^\prime+\varphi_2(x^\prime,\xi)
		\\
		&-\varphi_1(x,N(x,\xi))+Y(x,\xi)\cdot N(x,\xi)-\varphi_2(Y(x,\xi),\xi),
	\end{align*}
	which implies
	\begin{align*}
		\varphi_0(x,&Y(x,\xi)+y\cdot\omega(x,\xi)^{-1},N(x,\xi)+\eta\cdot\omega(x,\xi),\xi)
		\\
		=& \varphi_1(x,N(x,\xi)+\eta\cdot\omega(x,\xi))-
		(Y(x,\xi)+y\cdot\omega(x,\xi)^{-1})\cdot(N(x,\xi)+\eta\cdot\omega(x,\xi))
		\\
		+&\varphi_2(Y(x,\xi)+y\cdot\omega(x,\xi)^{-1},\xi)-\varphi_1(x,N(x,\xi))-\varphi_2(Y(x,\xi),\xi)
		+Y(x,\xi)\cdot N(x,\xi)
		\\
		=&-y\cdot\eta+
		(\varphi_1(x,N(x,\xi)+\eta\cdot\omega(x,\xi))-\varphi_1(x, N(x, \xi)))-Y(x,\xi)\cdot\eta\cdot\omega(x,\xi)
		\\
		&+(\varphi_2(Y(x,\xi)+y\cdot\omega(x,\xi)^{-1},\xi)-\varphi_2(Y(x,\xi),\xi))-y\cdot N(x,\xi)\cdot\omega(x,\xi)^{-1}.
	\end{align*}
	Then, recalling that $Y(x,\xi)=\varphi^\prime_{1\xi}(x,N(x,\xi))$ and $N(x,\xi)=\varphi^\prime_{2x}(Y(x,\xi),\xi)$, we get
\begin{align*}
	\varphi(y,\eta;x,\xi)=&-y\cdot\eta+(\varphi_1(x,N(x,\xi)+\eta\cdot\omega(x,\xi))-\varphi_1(x, N(x, \xi)))
	\\
	&-Y(x,\xi)\cdot\eta\cdot\omega(x,\xi)+(\varphi_2(Y(x,\xi)+y\cdot\omega(x,\xi)^{-1},\xi)-\varphi_2(Y(x,\xi),\xi))
	\\
	&-
	y\cdot N(x,\xi)\cdot\omega(x,\xi)^{-1}
	\\
	=&-y\cdot\eta+(\varphi_1(x,N(x,\xi)+\eta\cdot\omega(x,\xi))-\varphi_1(x, N(x, \xi)))
	\\
	&-\varphi^\prime_{1\xi}(x,N(x,\xi))\cdot\eta\cdot\omega(x,\xi)
	\\
	&+(\varphi_2(Y(x,\xi)+y\cdot\omega(x,\xi)^{-1},\xi)-\varphi_2(Y(x,\xi),\xi))
	\\
	&-\varphi^\prime_{2x}(Y(x,\xi),\xi))\cdot y\cdot\omega(x,\xi)^{-1}
	\\
	=&-y\cdot\eta+[A_1(\eta;x,\xi)\eta]\cdot\eta+[A_2(y;x,\xi)y]\cdot y,
\end{align*}
	that is \eqref{biss} and its subsequent expression in terms of $A_1, A_2$. Then \eqref{biss2} and \eqref{biss3} immediately follow taking derivatives with respect to $y,\eta$ in \eqref{biss}, and then looking at the definitions of $B_1$, $B_2$.
\end{proof}

\begin{lemma}\label{lem:abest}
	For $A_1,A_2,B_1,B_2$ defined in Lemma \ref{lem:ab} we have, for all $x,y,\xi,\eta\in\R^n$ in $\supp\,\rho$,
	\begin{align*}
		\|\partial_x^\beta\partial_\xi^\alpha\partial_\eta^{\alpha^\prime}(A_1,B_1)(\eta;x,\xi)\|
		&\lesssim
		\tau\norm{\xi}^{-|\alpha|-\frac{|\alpha^\prime|}{2}}\norm{x}^{-|\beta|-\frac{|\alpha^\prime|}{2}}
		\norm{y,\eta}^{|\alpha+\beta|},
	\\
		\|\partial_x^\beta\partial_\xi^\alpha\partial_y^{\beta^\prime}(A_2,B_2)(y;x,\xi)\|
		&\lesssim
		\tau\norm{\xi}^{-|\alpha|-\frac{|\beta^\prime|}{2}}\norm{x}^{-|\beta|-\frac{|\beta^\prime|}{2}}
		\norm{y,\eta}^{|\alpha+\beta|},
	\end{align*}
where $\langle y,\eta\rangle:=\sqrt{1+|y|^2+|\eta|^2},$ $y,\eta\in\R^n$.	
\end{lemma}
\begin{proof}
	The result follows from the Fa\'a di Bruno formula for the derivatives of the composed functions,
	the properties of $X\in S^{1,0}$, $ N\in S^{0,1}$ stated above,
	the fact that, on $\supp\,\rho$, \eqref{eq:equiv} holds for any $\theta\in[0,1]$, as well as
	\[
		Y(x,\xi)+\theta\cdot y\cdot\omega(x,\xi)^{-1}\in S^{1,0}\cdot\norm{y,\eta},
		\quad
		N(x,\xi)+\theta\cdot\eta\cdot\omega(x,\xi)\in S^{0,1}\cdot\norm{y,\eta},
	\]
	recalling that the seminorms of $J_1$ and $J_2$ involving their derivatives up to order 2
	are proportional to $\tau\in(0,1)$. 
	
	The proof works by induction on the order of the derivatives. Let us give an idea of the step $|\alpha+\beta+\alpha'|=1$. Let $e_j$ be the multiindex such that $|e_j|=1$, with components $0$ everywhere apart from the $j$-th. Then, for instance, on $\supp\,\rho$,
	\begin{align*}
		\partial^{e_j}_xB_1(\eta;x,\xi)&= (\partial^{e_j}_x\omega^2)
		\int_0^1 J^{\prime\prime}_{1\xi\xi}(\dots)\,d\theta
		+ \omega^2\int_0^1 J^{\prime\prime\prime}_{1x\xi\xi}(\dots)\,d\theta
		\\&+\omega^2\int_0^1 J^{\prime\prime\prime}_{1\xi\xi\xi}(\dots)\,d\theta
		\cdot\partial_x^{e_j}(N(x,\xi)+\theta\cdot\eta\cdot\omega(x,\xi))
		\\
		&\in S^{-1,0}+S^{-1,0}\cdot\norm{y,\eta}\subset S^{-1,0}\cdot\norm{y,\eta},
	\end{align*}
	since $\omega^2\in S^{-1,1}$, $\int_0^1 J^{\prime\prime}_{1\xi\xi}(\dots)\,d\theta\in S^{1,-1}$, $\int_0^1 J^{\prime\prime\prime}_{1x\xi\xi}(\dots)\,d\theta\in S^{0,-1}$, $\int_0^1 J^{\prime\prime\prime}_{1\xi\xi\xi}(\dots)\,d\theta\in S^{1,-2}$, and $N(x,\xi)+\theta\cdot\eta\cdot\omega(x,\xi)\in S^{0,1}|\eta|$.
	Similarly,
	\begin{align*}
		\partial^{e_j}_\xi B_1(\eta;x,\xi)&=(\partial^{e_j}_\xi\omega^2)
		\int_0^1 J^{\prime\prime}_{1\xi\xi}(\dots)\,d\theta
		\\&+\omega^2\int_0^1 J^{\prime\prime\prime}_{1\xi\xi\xi}(\dots)\,d\theta
		\cdot\partial_\xi^{e_j}(N(x,\xi)+\theta\cdot\eta\cdot\omega(x,\xi))
		\\
		&\in S^{0,-1}+S^{0,-1}\cdot\norm{y,\eta}\subset S^{0,-1}\cdot\norm{y,\eta},
		\\
		\partial^{e_j}_\eta B_1(\eta;x,\xi)&=\omega^2\int_0^1 J^{\prime\prime\prime}_{1\xi\xi\xi}(\dots)\,d\theta\cdot(\theta\cdot\omega(x,\xi)) \in S^{-1/2,-1/2}.
	\end{align*}
	The estimates for general multiindeces follow by induction.
\end{proof}
\begin{lemma}\label{lem:diffeq}
	On $\supp\,\rho$,
	\[
		|\varphi^\prime_y(y,\eta;x,\xi)|+|\varphi^\prime_\eta(y,\eta;x,\xi)|\asymp |y|+|\eta|.
	\]
\end{lemma}
\begin{proof}
	From Lemmas \ref{lem:ab} and \ref{lem:abest}, on $\supp\,\rho$, for $\tau\in(0,1)$,
	\begin{align*}
		\|B_1(\eta;x,\xi)\|\lesssim\tau &\Rightarrow \|B_1(\eta;x,\xi)\eta\|\lesssim\tau|\eta|,
		\\
		\|B_2(y;x,\xi)\|\lesssim\tau &\Rightarrow \|B_2(y;x,\xi)y\|\lesssim\tau|y|,
	\end{align*}
	which imply
	\begin{align*}
		|\varphi^\prime_y(y,\eta;x,\xi)|\lesssim|\eta|+\tau|y|, \quad & |\varphi^\prime_y(y,\eta;x,\xi)|\gtrsim|\eta|-\tau|y|,
		\\
		|\varphi^\prime_\eta(y,\eta;x,\xi)|\lesssim|y|+\tau|\eta|, \quad & |\varphi^\prime_\eta(y,\eta;x,\xi)|\gtrsim|y|-\tau|\eta|.
	\end{align*}
	These give
	\begin{align*}
		|\varphi^\prime_y(y,\eta;x,\xi)|+|\varphi^\prime_\eta(y,\eta;x,\xi)|&\lesssim(1+\tau)(|y|+|\eta|),
		\\
		|\varphi^\prime_y(y,\eta;x,\xi)|+|\varphi^\prime_\eta(y,\eta;x,\xi)|&\gtrsim(1-\tau)(|y|+|\eta|),
	\end{align*}
	as claimed.
\end{proof}
\begin{lemma}\label{lem:phiprimeest}
	On $\supp\,\rho$, for any multiindeces $\alpha,\beta,\alpha^\prime,\beta^\prime$, and all $x,y,\xi,\eta$,
	\begin{align*}
		|\partial_x^\beta\partial_y^{\beta^\prime}\partial_\xi^\alpha\partial_\eta^{\alpha^\prime}\varphi^\prime_y(y,\eta;x,\xi)|
		&\lesssim
		\begin{cases}
			0 & \text{if $|\alpha^\prime|\ge2$},
			\\
			1 & \text{if $|\alpha^\prime|=1$},
			\\
			\tau
			\norm{x}^{-|\beta|-\frac{|\beta^\prime|}{2}}  
			\norm{\xi}^{-|\alpha|-\frac{|\beta^\prime|}{2}}
			\norm{y,\eta}^{1+|\alpha+\beta|}
			& \text{if $|\alpha^\prime|=0$,}
			\\
			&\text{\phantom{if} $|\alpha+\beta+\beta^\prime|>0$};
		\end{cases}
		\\
		|\partial_x^\beta\partial_y^{\beta^\prime}\partial_\xi^\alpha\partial_\eta^{\alpha^\prime}
		\varphi^\prime_\eta(y,\eta;x,\xi)|
		&\lesssim
		\begin{cases}
			0 & \text{if $|\beta^\prime|\ge2$},
			\\
			1 & \text{if $|\beta^\prime|=1$},
			\\
			\tau
			\norm{x}^{-|\beta|-\frac{|\alpha^\prime|}{2}}  
			\norm{\xi}^{-|\alpha|-\frac{|\alpha^\prime|}{2}}
			\norm{y,\eta}^{1+|\alpha+\beta|}
			& \text{if $|\beta^\prime|=0$,}
			\\
			& \text{\phantom{if} $|\alpha+\alpha^\prime+\beta|>0$}.
		\end{cases}
	\end{align*}
\end{lemma}
\begin{proof}
	The results follow from Lemma \ref{lem:abest} and the estimates \eqref{eq:equiv}.
\end{proof}
\begin{lemma}\label{lem:phiprime_xxi}
	On $\supp\,\rho$, for any multiindeces $\alpha,\beta,\alpha^\prime,\beta^\prime$, and all $x,y,\xi,\eta$,
		\begin{align*}
		|\partial_x^\beta\partial_y^{\beta^\prime}\partial_\xi^\alpha\partial_\eta^{\alpha^\prime}\varphi^\prime_x(y,\eta;x,\xi)|
		&\lesssim
		\begin{cases}
			\tau
			\norm{\xi}^{-|\alpha|-\frac{|\beta^\prime|}{2}}
			\norm{x}^{-1-|\beta|-\frac{|\beta^\prime|}{2}}
			\norm{y,\eta}^{3+|\alpha+\beta|}
			 & \text{if $|\beta^\prime|>0$},
			\\
			\tau
			\norm{\xi}^{-|\alpha|-\frac{|\alpha^\prime|}{2}}
			\norm{x}^{-1-|\beta|-\frac{|\alpha^\prime|}{2}}
			\norm{y,\eta}^{3+|\alpha+\beta|}
			& \text{if $|\alpha^\prime|>0$},
			\\
			\tau
			\norm{\xi}^{-|\alpha|}
			\norm{x}^{-1-|\beta|}
			\norm{y,\eta}^{3+|\alpha+\beta|}
			& \text{if $\alpha^\prime=\beta^\prime=0$};
		\end{cases}
		\\
		|\partial_x^\beta\partial_y^{\beta^\prime}\partial_\xi^\alpha\partial_\eta^{\alpha^\prime}
		\varphi^\prime_\xi(y,\eta;x,\xi)|
		&\lesssim
		\begin{cases}
			\tau
			\norm{\xi}^{-1-|\alpha|-\frac{|\beta^\prime|}{2}}
			\norm{x}^{-|\beta|-\frac{|\beta^\prime|}{2}}
			\norm{y,\eta}^{3+|\alpha+\beta|}
			 & \text{if $|\beta^\prime|>0$},
			\\
			\tau
			\norm{\xi}^{-1-|\alpha|-\frac{|\alpha^\prime|}{2}}
			\norm{x}^{-|\beta|-\frac{|\alpha^\prime|}{2}}
			\norm{y,\eta}^{3+|\alpha+\beta|}
			& \text{if $|\alpha^\prime|>0$},
			\\
			\tau
			\norm{\xi}^{-1-|\alpha|}
			\norm{x}^{-|\beta|}
			\norm{y,\eta}^{3+|\alpha+\beta|}
			& \text{if $\alpha^\prime=\beta^\prime=0$}.
		\end{cases}
	\end{align*}
\end{lemma}
\begin{proof}
	The results follow from Lemma \ref{lem:abest}, observing that
	\begin{align*}
		\varphi^\prime_x(y,\eta;x,\xi)&=
		d_x[(A_1(\eta;x,\xi)\eta)\cdot\eta]+d_x[(A_2(y;x,\xi)y)\cdot y],
		\\
		\varphi^\prime_\xi(y,\eta;x,\xi)&=
		d_\xi[(A_1(\eta;x,\xi)\eta)\cdot\eta]+d_\xi[(A_2(y;x,\xi)y)\cdot y].
	\end{align*}
\end{proof}
\begin{lemma}\label{lem:rho}
		For any multiindeces $\alpha,\beta,\alpha^\prime,\beta^\prime$, and all $x,y,\xi,\eta$,
		\[
			|\partial^\alpha_\xi\partial^{\alpha^\prime}_\eta\partial^\beta_x\partial^{\beta^\prime}_y
			\rho(y,\eta;x,\xi)|\lesssim
			\norm{\xi}^{-|\alpha|-\frac{|\alpha^\prime|}{2}}
			\norm{x}^{-|\beta|-\frac{|\beta^\prime|}{2}}.
		\]
\end{lemma}
\begin{proof}
	Immediate, by the definition of $\rho$, the hypotheses on $\psi$, the properties
	$Y(x,\xi)\in S^{1,0}$, $N(x,\xi)\in S^{0,1}$, and the estimates \eqref{eq:equiv}.
\end{proof}
\begin{lemma}\label{lem:Gamma}
	Let
	\[
		\Gamma=\Gamma(y,\eta;x,\xi)=1+|\varphi^\prime_y(y,\eta;x,\xi)|^2+|\varphi^\prime_\eta(y,\eta;x,\xi)|^2.
	\]
	Then, on $\supp\,\rho$, 
	for any multiindeces $\alpha,\beta,\alpha^\prime,\beta^\prime$, and all $x,y,\xi,\eta$,
	\[
		\left|
		\partial^\alpha_\xi\partial^{\alpha^\prime}_\eta\partial^\beta_x\partial^{\beta^\prime}_y
		\left(\frac{1}{\Gamma(y,\eta;x,\xi}\right)\right|
		\lesssim
		\tau\norm{\xi}^{-|\alpha|}\norm{x}^{-|\beta|}\norm{y,\eta}^{-2+|\alpha+\beta|}.
	\]
\end{lemma}
\begin{proof}
	Immediate, by Lemmas \ref{lem:ab}, \ref{lem:abest}, \ref{lem:diffeq}, and \ref{lem:phiprimeest}.
\end{proof}
The next Lemma \ref{lem:tL} is a straightforward consequence of Lemma \ref{lem:Gamma} and the definition of transpose
operator.
\begin{lemma}\label{lem:tL}
	Let us define the operator
	\[
		M=\frac{1}{\Gamma}(1-i\varphi^\prime_y(y,\eta;x,\xi)\cdot\nabla_y-i\varphi^\prime_\eta(y,\eta;x,\xi)\cdot\nabla_\eta)\]
		such that $Me^{i\varphi(y,\eta;x,\xi)}=e^{i\varphi(y,\eta;x,\xi)}.$	%
	Then, 
	\[
		{^t}M=M_0+M_1\cdot\nabla_y+M_2\cdot\nabla_\eta,
	\]
	where, on $\supp\,\rho$, 
	for any multiindeces $\alpha,\beta,\alpha^\prime,\beta^\prime$, and all $x,y,\xi,\eta$,
	\[
		\|	
		\partial^\alpha_\xi\partial^{\alpha^\prime}_\eta\partial^\beta_x\partial^{\beta^\prime}_y
		[(M_0,M_1,M_2)(y,\eta;x,\xi)]
		\|
		\lesssim
		\norm{\xi}^{-|\alpha|}\norm{x}^{-|\beta|}\norm{y,\eta}^{-1+|\alpha+\beta|}.
	\]
\end{lemma}
\begin{proof}[Proof of Proposition \ref{prop:p0}]
	Using the operator $M$ defined in Lemma \ref{lem:tL}, we have, for arbitrary $k\in\Z_+$,
	\[
		p_0(x,\xi)=\iint e^{i\varphi(y,\eta;x,\xi)}(({^t}M)^k \rho)(y,\eta;x,\xi)\,dy\dcut\eta.
	\]
	Notice that, from the analysis above, 
	for any $k\in\Z_+$, any multiindeces $\alpha^\prime,\beta^\prime$, and all $x,y,\xi,\eta$,
	\[
		|(\partial^{\alpha^\prime}_\xi\partial^{\beta^\prime}_x(({^t}M)^k \rho))(y,\eta;x,\xi)|
		\lesssim \norm{x}^{-|\beta^\prime|}\norm{\xi}^{-|\alpha^\prime|}\norm{y,\eta}^{-k+|\alpha^\prime+\beta^\prime|}.
	\]
	Then, for any fixed $\alpha,\beta\in\Z_+^n$, and arbitrary $k\in\Z_+$, we find %
	\begin{align*}
		&\partial^\alpha_\xi\partial^\beta_x p_0(x,\xi)=
		\\
		&=\!\!\!\!\!\sum_{\alpha_1+\alpha_2=\alpha}\sum_{\beta_1+\beta_2=\beta}
		\binom{\alpha}{\alpha_1}
		\binom{\beta}{\beta_1}\!
		\iint \!\!
		\left(\partial^{\alpha_1}_\xi\partial^{\beta_1}_x e^{i\varphi(y,\eta;x,\xi)}\right)\!
		\cdot\!(\partial^{\alpha_2}_\xi\partial^{\beta_2}_x(({^t}M)^k \rho)(y,\eta;x,\xi))\,dy\dcut\eta.
	\end{align*}
	Choosing $k$ such that $-k+6|\alpha+\beta|\le-(2n+1)$, from the results in Lemmas \ref{lem:phiprime_xxi},
	\ref{lem:rho}, and \ref{lem:tL} above, we get
	\[
		|\partial^\alpha_x\partial^\beta_\xi p_0(x,\xi)|
		\lesssim 
		\norm{x}^{-|\alpha|}\norm{\xi}^{-|\beta|}\iint\norm{y,\eta}^{-(2n+1)}\,dyd\eta
		\lesssim
		\norm{x}^{-|\alpha|}\norm{\xi}^{-|\beta|},
	\]
	as claimed.
\end{proof}
\begin{remark}
Let us notice that we have proved here above that the seminorms of $p_0$ are controlled by those of $\varphi_1$ and $\varphi_2$. This implies that,  if $J_1$ and $J_2$ are bounded in $S^{1,1}$, so is $p_0$ in $S^{0,0}.$ The boundedness conditions of Theorem \ref{thm:mainbis} are so fulfilled, and the proof of Theorem \ref{thm:mainbis} is complete.\end{remark}

\section{Fundamental solution to hyperbolic systems in SG classes}\label{sec:fundsol}
\setcounter{equation}{0}
%
%
In the present section we apply the results of Sections \ref{sec:mpsgphf} and \ref{sec:sgfioprod} to construct the fundamental solution $E(t,s)$ to the Cauchy problem for a first order system of partial differential equations of hyperbolic type, with coefficients in SG classes and roots of (possibly) variable multiplicity. A standard argument, which we omit here, gives then the solution, via $E(t,s)$ and Duhamel's formula, see Theorem \ref{wp} below. We follow the approach in \cite[Section 10.7]{Kumano-go:1}. 

Let us consider the Cauchy problem
\begin{equation}\label{sys}
\begin{cases}
	LW(t,x) = F(t,x) & (t,x)\in (0,T]\times\R^n, \\
	W(0,x)  = W_0(x)  & x\in \R^n,
\end{cases}
\end{equation}
where
\begin{equation}\label{L}
L(t,x,D_t,D_x) = D_t + \Lambda(t,x,D_x) +R(t,x,D_x),
\end{equation}
$\Lambda$ is an $m\times m$ diagonal operator matrix whose entries $\lambda_j(t,x,D_x)$, $j=1,\dots, m$, are pseudo-differential operators with symbols $\lambda_j(t,x,\xi)\in C([0,T]; S^{\epsilon,1})$, $\epsilon\in[0,1]$, and $R$ is an $m\times m$-operator matrix with elements in $C([0,T],S^{\epsilon-1,0})$. The case $\epsilon=0$ corresponds to symbols uniformly bounded in the space variable, while the case $\epsilon=1$ is the standard situation of SG symbols with equal order components.

Assume also that the system \eqref{L} is of hyperbolic type, that is, $\lambda_j(t,x,\xi)\in\R$, $j=1,\dots,m$. Notice that, differently
from \cite{Coriasco:998.2, Coriasco:998.3}, here
we do not impose any ``separation condition at infinity'' on the $\lambda_j$, $j=1,\dots,m$.
Indeed, the results presented below apply both to the constant as well as the variable multiplicities cases.

For $0<T_0\leq T$, we define $\Delta_{T_0}:=\{(t,s)\vert\ 0\leq s\leq t\leq T_0\}$. The fundamental solution of 
\eqref{sys} is a family $\{E(t,s)\vert (t,s)\in \Delta_{T_0}\}$ of SG FIOs, satisfying 
\begin{equation}\label{tocheck}
	\begin{cases}
		LE(t,s) = 0 	& (t,s)\in \Delta_{T_0},\\
		E(s,s)=I						& s\in[0,T_0].
	\end{cases}
\end{equation}
In this section we aim to show that, if $T_0$ is small enough, it is possible to construct the family $\{E(t,s)\}$ satisfying \eqref{tocheck}.

As a consequence of \eqref{tocheck}, it is quite easy to get the following:
\begin{theorem}\label{wp}
For every $F\in C([0,T]; H^{r,\varrho}(\R^n))$ and $G\in H^{r,\varrho}(\R^n)$, the solution $W(t,x)$ of the Cauchy problem \eqref{sys} exists uniquely, it belongs to the class $C([0,T_0], H^{r-(\epsilon-1),\varrho}(\R^n))$, and it is given by
$$W(t)=E(t,0)G+i\ds\int_0^t E(t,s)F(s)ds,\quad t\in[0,T_0].$$
\end{theorem}
\begin{remark}
Theorem \ref{wp} gives well-posedness of the Cauchy problem \eqref{sys} in $\mathcal S(\R^n)$ and $\mathcal S'(\R^n)$;  moreover it gives "well posedness with loss/gain of decay" (depending on the sign of $r$) of \eqref{sys} in weighted Sobolev spaces $H^{r,\varrho}(\R^n).$ This phenomenon is quite common in the theory of hyperbolic partial
differential equations with SG type coefficients, see \cite{AC06,AC10,AC13}.
We remark that in the symmetric case $\epsilon=1$ the Cauchy problem \eqref{sys} turns out to be well-posed also in $H^{r,\varrho}(\R^n).$
\end{remark}

To begin, consider SG phase functions $\varphi_j = \varphi_j(t,s,x,\xi),\ 1\leq j\leq m,$ defined on $\Delta_{T_0}\times\R^{2n}$, and define the operator matrix
\[ I_\varphi(t,s) = \begin{pmatrix} I_{\varphi_1}(t,s) & & 0 \\ & \ddots & \\ 0 & & I_{\varphi_m}(t,s)\end{pmatrix}, \]
where $I_{\varphi_j}:=Op_{\varphi_j}(1),$ $1\leq j\leq m$. From Theorem \ref{thm:compi} (see Remark \ref{rem:asymp}) we see that
\beqsn
D_t I_{\varphi_j} + \lambda_j(t,x,D_x)I_{\varphi_j}& = &\int e^{i\varphi_j(t,s,x,\xi)}\frac{\partial\varphi_j}{\partial t}(t,s,x,\xi)\dbar\xi 
\\
&+& \int e^{i\varphi_j(t,s,x,\xi)} \lambda_j(t,x,\varphi'_{j,x}(t,s,x,\xi))\dbar\xi 
\\
&+& \int e^{i\varphi_j(t,s,x,\xi)} b_{0,j}(t,s,x,\xi) \dbar\xi,
\eeqsn
where $b_{0,j}(t,s)\in S^{\epsilon-1,0}\subseteq S^{0,0}$. The first two integrals in the right-hand side of the equation here above cancel if we choose $\varphi_j$, $j=1,\dots,m$, to be the solution to the eikonal equation \eqref{eik} associated with the symbol $a=\lambda_j$, $j=1,\dots,m$. By Proposition \ref{trovala!}, this is possible, provided that $T_0$ is small enough. Writing $B_{0,j}:=Op_{\varphi_j}(b_{0,j})$, we define the family $\{W_1(t,s);(t,s)\in\Delta_{T_0}\}$ of SG FIOs by
\beqsn
W_1(t,s,x,D_x):= -i\left( \begin{pmatrix} B_{0,1}(t,s,x,D_x) & & 0 \\ & \ddots & \\ 0 & & B_{0,m}(t,s,x,D_x)\end{pmatrix} + R(t,x,D_x)\right)I_{\varphi}(t,s,x,D_x),
\eeqsn
and we denote by $w_1(t,s,x,\xi)$ the symbol of $W_1(t,s,x,\xi).$ Notice that
\begin{equation}\label{heart}
L(t,x,D_x)I_{\varphi}(t,s,x,D_x) = i W_1(t,s,x,D_x),
\end{equation}
that is, $iW_1$ is the residual of system \eqref{sys} for $I_\varphi$. We define then by induction the sequence of $m\times m$-matrices of SG FIOs, denoted by $\{W_\nu(t,s);(t,s)\in\Delta_{T_0}\}_{\nu\in\N}$, as 
\begin{equation}\label{eq:Wn+1}
  W_{\nu+1}(t,s,x,D_x) = \int_s^t W_1(t,\theta,x,D_x)W_\nu(\theta,s,x,D_x)d\theta,
\end{equation}
and we denote by $w_{\nu+1}(t,s,x,\xi)$ the symbol of $W_{\nu+1}(t,s,x,D_x).$
We are now going to prove that the operator norms of $W_\nu$, seen as operators from the Sobolev space $H^{r,\varrho}$ into $H^{r-(\nu-1)(\epsilon-1),\varrho}$ for any fixed $(r,\varrho)\in\R^2$ can be estimated from above by
\begin{equation}\label{claim}
\|W_\nu(t,s)\|_{\mathcal L(H^{r,\varrho},H^{r-(\nu-1)(\epsilon-1),\varrho})}\leq \frac{C_{r,\varrho}^{\nu-1}|t-s|^{\nu-1}}{(\nu-1)!} \leq \frac{C_{r,\varrho}^{\nu-1}T_0^{\nu-1}}{(\nu-1)!},
\end{equation}
for all $(t,s)\in\Delta_{T_0}$ and $\nu\in\N$, where $C_{r,\varrho}$ is a constant which only depends on $r,\varrho$.

To deal with the operator norms in \eqref{claim}, we need to explicitly write the matrices $W_\nu$. An induction in \eqref{eq:Wn+1} easily shows that
\begin{equation}\label{eq:wn+12}
	W_\nu(t,s) = \int_s^t\int_s^{\theta_1}\ldots\int_s^{\theta_{\nu-2}} W_1(t,\theta_1)\ldots W_1(\theta_{\nu-2},\theta_{\nu-1})d\theta_{\nu-1}\ldots d\theta_1.
\end{equation}
The integrand is a product of $\nu-1$ $m\times m$-matrices  of SG FIOs, therefore it is an operator matrix whose entries consist of $m^{\nu-2}$ summands of compositions of $\nu-1$ SG FIOs. Denoting by $Q_1\circ\ldots\circ Q_{\nu-1}$ one of these compositions, where each of the $Q_j$ is one of the $m^2$ entries of the $m\times m$-matrix of SG FIOs $W_1$, we have from Example \ref{32} and (2) of Theorem \ref{thm:main} that $Q_1\circ\ldots\circ Q_{\nu-1}$ is again a SG FIO with symbol $q_{1,\ldots,\nu-1}\in S^{(\nu-1)(\epsilon-1),0}\subseteq S^{0,0}$. Moreover, from (3) of Theorem \ref{thm:main}, for all $\ell\in\N$ there exists $C_\ell>0$ and $\ell'\in\N_0$ such that
\beqsn
&&|||q_{1,\ldots,\nu-1}(t,\theta_1,\ldots,\theta_{\nu-1})|||_{\ell}^{(\nu-1)(\epsilon-1),0}
\\
&&\hskip+3cm \leq C_\ell^{\nu-2} |||q_1(t,\theta_1)|||_{\ell'}^{\epsilon-1,0} \ldots |||q_{\nu-1}(\theta_{\nu-2},\theta_{\nu-1})|||_{\ell'}^{\epsilon-1,0},
\eeqsn 
where for $j=1,\ldots,\nu-1$, $q_j(t,s)$ denotes the symbol of the SG FIO $Q_j(t,s)$, $(t,s)\in\Delta_{T_0}$.
Now we set
\[ \bar{\sigma} := \sup_{j=1,\ldots,\nu-1}\sup_{(t,s)\in\Delta_{T_0}} |||q_j(t,s)|||_{\ell'}^{\epsilon-1,0} < \infty, \]
so that
\[ |||q_{1,\ldots,\nu-1}(t,\theta_1,\ldots,\theta_{\nu-1})|||_{\ell}^{(\nu-1)(\epsilon-1),0} \leq C_\ell^{\nu-2} \bar{\sigma}^{\nu-1}.\]
The continuity of the SG FIOs $Q_1\circ\ldots \circ Q_{\nu-1}(t,\theta_1,\cdots,\theta_{\nu-1}):H^{r,\varrho}\longrightarrow H^{r-(n-1)(\epsilon-1),\varrho}$ (see Theorem \ref{thm:sobcont}) and the previous inequality give that for every $r,\varrho$ there exist constants $C_{r,\varrho}>0$ (depending only on the indeces of the Sobolev space) and $\ell_{r,\varrho}\in\N_0$ such that for all $u\in H^{r,\varrho}$
\beqs\label{eq:calderonvaillancourt}
&&\|Q_1(t,\theta_1)\circ\ldots\circ Q_{\nu-1}(\theta_{\nu-2}\theta_{\nu-1})u\|_{r-(n-1)(\epsilon-1),\varrho}
\\
\nonumber
&&\hskip+3cm \leq C_{r,\varrho}|||q_{1,\ldots,\nu-1}(t,\theta_1,\ldots,\theta_{\nu-1})|||^{(\nu-1)(\epsilon-1),0}_{\ell_{r,\varrho}}\|u\|_{r,\varrho}
\\
\nonumber
&&\hskip+3cm \leq C_{r,\varrho} C_{\ell_{r,\varrho}}^{\nu-2} \bar{\sigma}^{\nu-1}\|u\|_{r,\varrho}.
\eeqs
Therefore, in the operator matrix $W_1(t,\theta_1)\ldots W_1(\theta_{\nu-2},\theta_{\nu-1})$, the operator norm of each entry can be bounded from above by $m^{\nu-2}C_{r,\varrho}C_{\ell_{r,\varrho}}^{\nu-2}\bar{\sigma}^{\nu-1}$. Now by \eqref{eq:wn+12} and \eqref{eq:calderonvaillancourt} we deduce that
\beqs
\nonumber
&&\|W_\nu(t,s)\|_{\mathcal L(H^{r,\varrho},H^{r-(\nu-1)(\epsilon-1),\varrho})}
\\
\nonumber
&&\hskip+1cm \leq \int_s^t\int_s^{\theta_1}\ldots\int_s^{\theta_{\nu-2}} \|W_1(t,\theta_1)\ldots W_1(\theta_{\nu-2},\theta_{\nu-1})
  \|_{\mathcal L(H^{r,\varrho},H^{r-(\nu-1)(\epsilon-1),\varrho})} d\theta_{\nu-1}\ldots d\theta_1 
 \\\nonumber
 &&\hskip+1cm 
 \leq m^{\nu-2}C_{r,\varrho}C_{\ell_{r,\varrho}}^{\nu-2}\bar{\sigma}^{\nu-1}\int_s^t\int_s^{\theta_1}\ldots\int_s^{\theta_{\nu-2}} d\theta_{\nu-1}\ldots d\theta_1 
 \\\label{eq:normWn}
&&\hskip+1cm \leq \frac{m^{\nu-2}C_{r,\varrho}C_{\ell_{r,\varrho}}^{\nu-2}\bar{\sigma}^{\nu-1}|t-s|^{\nu-1}}{(\nu-1)!}= \frac{\tilde C_{r,\varrho}^{\nu-1}|t-s|^{\nu-1}}{(\nu-1)!}
\eeqs
for a new constant $\tilde C_{r,\varrho}$ depending only on $r,\varrho$, which yields the claim \eqref{claim}.

Now, using the estimate \eqref{claim}, we can show that the sequence of SG FIOs, defined for all $(t,s)\in\Delta_{T_0}$ and all $N\in\N$ by
\begin{equation}\label{eq:En}
	E_N(t,s) = I_\varphi(t,s) + \int_s^t I_\varphi(t,\theta)\sum_{\nu=1}^N W_\nu(\theta,s)d\theta,
\end{equation}
is a well-defined SG FIO in $\mathcal L(H^{r,\varrho},H^{r-\epsilon+1,\varrho})$ for every $r,\varrho$, and converges, as $N\to\infty$, to the well-defined SG FIO,
belonging to $\mathcal L(H^{r,\varrho},H^{r-\epsilon+1,\varrho})$, given by
\begin{equation}\label{eq:E}
	E(t,s) = I_\varphi(t,s) + \int_s^t I_\varphi(t,\theta)\sum_{\nu=1}^\infty W_\nu(\theta,s)d\theta.
\end{equation}
$E(t,s)$ in \eqref{eq:E} is the fundamental solution to the system \eqref{sys} in the sense that it satisfies \eqref{tocheck}. Indeed, at symbols level, with the notations $E_N=Op_\varphi(e_N)$, $E=Op_\varphi(e)$ and $W_1\circ\cdots \circ W_1= Op_\varphi(\sigma_{\nu-1})$,  for every $l\in\N$ and $|\alpha+\beta|\leq \ell$, we have
\begin{align*}\label{eq:boundsforsigma}
  & |\partial^\alpha_\xi\partial^\beta_x e_N(t,s,x,\xi)| \\
  & \leq \int_s^t \sum_{\nu=1}^N |\partial^\alpha_\xi\partial^\beta_x w_\nu(\theta,s,x,\xi)|d\theta\\
  & \leq \sum_{\nu=1}^N\int_s^t\int_s^\theta\int_s^{\theta_1}\hskip-0.4cm\ldots\int_s^{\theta_{\nu-2}} \big|\partial^\alpha_\xi\partial^\beta_x \sigma_{\nu-1}(t,\theta_1,\ldots,\theta_{\nu-1},x,\xi)\big|  d\theta_{n-1}\ldots d\theta_1d\theta \\
  & \leq \sum_{\nu=1}^N\int_s^t\ldots\int_s^{\theta_{\nu-2}} |||\sigma_{\nu-1}(t,\theta_1,\ldots,\theta_{\nu-1})|||_\ell^{(\nu-1)(\epsilon-1),0}\x^{(\nu-1)(\epsilon-1)-|\beta|}\langle \xi\rangle^{-|\alpha|} d\theta_{\nu-1}\ldots d\theta\\
  & \leq \x^{\epsilon-1-|\beta|}\langle \xi\rangle^{-|\alpha|} \sum_{\nu=1}^N \frac{m^{\nu-2}C_{\ell}^{\nu-2}\bar{\sigma}^{\nu-1}|t-s|^{\nu-1}}{(\nu-1)!},
\end{align*}
so
\[ |||e_N(t,s)|||^{\epsilon-1,0}_{\ell} \leq \sum_{\nu=0}^{N-1} \frac{(C'_{\ell}|t-s|)^\nu}{\nu!}, \]
for a new constant $C'_{\ell}>0$. Then, for $N\to\infty$ we get
\[ |||e(t,s)|||^{\epsilon-1,0}_\ell \leq \exp(C'_{\ell}(t-s)) < \infty. \]
Thus, the SG FIO \eqref{eq:E} has a well-defined symbol. On the other hand, at operator's level, by definitions \eqref{eq:En} and \eqref{L} we have
\beqs\label{vce}
  LE_N = LI_\phi - i\sum_{\nu=1}^N W_\nu(t,s) + \int_s^t LI_\phi(t,\theta)\sum_{\nu=1}^N W_\nu(\theta,s)d\theta.
\eeqs
An induction shows that
\begin{equation}\label{eq:inductionW}
  \sum_{\nu=1}^N W_\nu(t,s) = -i(LI_\phi)(t,s) - i\int_s^t (LI_\phi)(t,\theta)\sum_{\nu=1}^{N-1} W_\nu(\theta,s)d\theta.
\end{equation}
Indeed, for $N=2$ we have by \eqref{heart} and \eqref{eq:Wn+1}
\[ W_1(t,s)+W_2(t,s) = -i(LI_\phi)(t,s) - i\int_s^t (LI_\phi)(t,\theta)W_1(\theta,s)d\theta;\]
the induction step $N\mapsto N+1$ works as follows:
\beqsn
  \sum_{\nu=1}^{N+1}W_\nu(t,s)
  & =& W_{N+1}(t,s) + \sum_{\nu=1}^{N} W_\nu(t,s) \\
  & =& -i\int_s^t (LI_\phi)(t,\theta)W_N(\theta,s)d\theta - i(LI_\phi)(t,s)- i\int_s^t (LI_\phi)(t,\theta)\sum_{\nu=1}^{N-1}W_\nu(\theta,s)d\theta \\
  & =& - i(LI_\phi)(t,s) - i\int_s^t (LI_\phi)(t,\theta)\sum_{\nu=1}^{N}W_\nu(\theta,s)d\theta.
\eeqsn
Substituting \eqref{eq:inductionW} into \eqref{vce} we get
\beqsn  (LE_N)(t,s) = \int_s^t (LI_\phi)(t,\theta)W_N(\theta,s)d\theta.
\eeqsn
Now, for $N\to\infty$, $\|W_N(t,s)\|_{\mathcal L(H^{r,\varrho},H^{r-(N-1)(\epsilon-1),\varrho})}\to0$ because of \eqref{eq:normWn}; thus $LE_N\to LE=0$. Moreover, it's easy to verify that $E(s,s)=I$. So, \eqref{tocheck} is fulfilled, and we have constructed the fundamental solution to $L$.
As it concerns the dependence of the fundamental solution on the parameters $(t,s)$, we finally notice that the SG FIO-valued map $(t,s)\mapsto E(t,s)$ belongs to $C(\Delta_{T_0})$, since $E$ is obtained by continuous operations of operators which are continuous in $t,s$, see \eqref{eq:E}.


\bibliographystyle{abbrv}

\end{document}